\documentclass{amsart}
\pdfoutput=1

\usepackage{amsthm,amssymb,amsmath,array,microtype,booktabs,mathtools}

\usepackage[pdftex,colorlinks,bookmarks=false,ocgcolorlinks]{hyperref}
\def\arXiv#1{arXiv:\href{http://arXiv.org/abs/#1}{#1}}
\usepackage{url}
\hypersetup{citecolor=[rgb]{0,0,0.6},linkcolor=[rgb]{0,0.3,0},urlcolor=[rgb]{0,0,0.6}}

\newcommand{\caseup}[1]{\textup{#1}}

\newcommand{\sC}{\mathcal{C}}
\newcommand{\C}{\mathbb{C}}

\newcommand{\R}{\mathbb{R}}
\newcommand{\N}{\mathbb{N}}
\newcommand{\Z}{\mathbb{Z}}
\newcommand{\F}{{\mathbb F}}
\newcommand{\X}{{\mathfrak X}}
\newcommand{\OO}{{\mathbb O}}
\newcommand{\HH}{{\mathbb H}}
\newcommand{\Proj}{{\mathbb P}}
\newcommand{\id}{\mathrm{id}}
\newcommand{\Span}{\operatorname{span}}
\newcommand{\rank}{\operatorname{rank}}

\makeatletter
\def\mathcenterto#1#2{\mathclap{\phantom{#1}\mathclap{#2}}\phantom{#1}}
\let\old@widetilde\widetilde
\def\widetildeto#1#2{\mathcenterto{#2}{\old@widetilde{\mathcenterto{#1}{#2\,}}}}
\makeatother

\newcommand{\Mtilde}{\widetildeto{X}{M}}

\newcommand\Cplusplus{\textsc{C}\nolinebreak[4]\raisebox{.4ex}{\tiny\textbf{++}}}

\newcommand{\sdim}[2]{r(#1,#2)}

\DeclareMathOperator{\Tr}{Tr}
\DeclareMathOperator{\PSL}{PSL}
\DeclareMathOperator{\RealPart}{Re}

\DeclareMathOperator{\SO}{SO}
\DeclareMathOperator*{\hilbertsum}{\widehat{\bigoplus}}

\theoremstyle{plain}
\newtheorem{theorem}{Theorem}
\newtheorem{corollary}[theorem]{Corollary}
\newtheorem{lemma}[theorem]{Lemma}
\newtheorem{proposition}[theorem]{Proposition}
\newtheorem{conjecture}[theorem]{Conjecture}

\numberwithin{theorem}{section}

\theoremstyle{definition}
\newtheorem{definition}[theorem]{Definition}
\newtheorem{remark}[theorem]{Remark}

\numberwithin{equation}{section}
\numberwithin{table}{section}
\numberwithin{figure}{section}

\begin{document}

\title{Optimal simplices and codes in projective spaces}

\author[Cohn]{Henry Cohn}
\address{Microsoft Research\\
One Memorial Drive\\
Cambridge, MA 02142} \email{cohn@microsoft.com}

\author[Kumar]{Abhinav Kumar}
\address{Department of Mathematics\\
Massachusetts Institute of Technology\\
Cambridge, MA 02139}
\curraddr{Department of Mathematics\\
Stony Brook University\\
Stony Brook, NY 11794}
\email{abhinavk@alum.mit.edu}

\author[Minton]{Gregory Minton}
\address{Microsoft Research\\
One Memorial Drive\\
Cambridge, MA 02142} \email{gminton@alum.mit.edu}

\thanks{AK was supported in part by National Science Foundation grants
  DMS-0757765 and DMS-0952486 and by a grant from the Solomon
  Buchsbaum Research Fund.  GM was supported by a Fannie and John
  Hertz Foundation Graduate Fellowship, a National Science Foundation
  Graduate Research Fellowship, and internships at Microsoft Research.
  An earlier version of this paper appears in Chapter~II of GM's
  doctoral dissertation.}

\date{September 20, 2015}

\begin{abstract}
We find many \emph{tight} codes in compact spaces, i.e.,
optimal codes whose optimality follows from linear programming
bounds. In particular, we show the existence (and abundance) of
several hitherto unknown families of simplices in quaternionic
projective spaces and the octonionic projective plane.  The
most noteworthy cases are $15$-point simplices in $\HH \Proj^2$
and $27$-point simplices in $\OO \Proj^2$, both of which are
the largest simplices and the smallest $2$-designs possible in
their respective spaces. These codes are all universally
optimal, by a theorem of Cohn and Kumar.  We also show the
existence of several positive-dimensional families of simplices
in the Grassmannians of subspaces of $\R^n$ with $n \leq 8$;
close numerical approximations to these families had been found
by Conway, Hardin, and Sloane, but no proof of existence was
known. Our existence proofs are computer-assisted, and the main
tool is a variant of the Newton-Kantorovich theorem.  This
effective implicit function theorem shows, in favorable
conditions, that every approximate solution to a set of
polynomial equations has a nearby exact solution. Finally, we
also exhibit a few explicit codes, including a configuration of
$39$ points in $\OO \Proj^2$ that form a maximal system of mutually
unbiased bases.  This is the last tight code in $\OO\Proj^2$
whose existence had been previously conjectured but not
resolved.
\end{abstract}

\maketitle

\tableofcontents

\section{Introduction}

The study of codes in spaces such as spheres, projective spaces, and
Grassmannians has been the focus of much interest recently, involving an
interplay of methods from many aspects of mathematics, physics, and computer
science \cite{BV, BNOV09, BRV13, BG09, R05, MQKF13, FKM13}. Given a compact
metric space $X$, the basic question is how to arrange $N$ points in $X$ so
as to maximize the minimal distance between them. A point configuration is
called a \emph{code}, and an \emph{optimal code} $\sC$ maximizes the minimal
distance between its points given its size $|\sC|$.  Finding optimal codes is
a central problem in coding theory.  Even when $X$ is finite (for example,
the cube $\{0,1\}^n$ under Hamming distance), this optimization problem is
generally intractable, and it is even more difficult when $X$ is infinite.

Most of the known optimality theorems have been proved using linear
programming bounds, and we are especially interested in codes for which these
bounds are sharp. We call them \emph{tight} codes.\footnote{The word
``tight'' is used for a related but more restrictive concept in the theory of
designs.  We use the same word here for lack of a good substitute. This makes
``tight'' a noncompositional adjective, much like ``optimal'': codes and
designs are both just sets of points, so every code is a design and vice
versa, but a tight code is not necessarily a tight design. (However, one can
show that every tight design is a tight code.)} The class of tight codes
includes many of the most remarkable codes known, such as the icosahedron and
the $E_8$ root system.

In this paper, we explore the landscape of tight codes in projective spaces.
We devote most of our attention to regular simplices (i.e., collections of
equidistant points). Tight simplices correspond to tight equiangular frames
\cite{STDH}, which have applications in signal processing and sparse
approximation, and they also capture interesting invariants of their ambient
spaces.  Furthermore, they seem to be by far the most widespread sort of
tight codes.

In real and complex projective spaces, tight simplices occur only
sporadically.  All known constructions are based on geometric,
group-theoretic, or combinatorial properties that depend delicately on the
number of points and the dimension of the projective space. By contrast, we
find a surprising new phenomenon in quaternionic and octonionic spaces: in
each dimension, there are substantial intervals of sizes for which tight
simplices always seem to exist. For instance, in $\HH \Proj^2$ we show that
tight simplices exist for $N$ points with $1 \leq N \leq 13$ or $N = 15$,
while in $\HH \Proj^3$ we show existence for $1 \leq N \leq 21$.

This behavior cannot plausibly be explained using the sorts of constructions
that work in real and complex spaces.  In fact, the new tight simplices
exhibit little structure and seem to exist not for any special reason, but
rather because of parameter counting: they can be characterized by systems of
equations with more variables than constraints.  Making this heuristic
precise, and indeed extracting any proof from this approach, requires a
delicate choice of constraints. Much of our paper is devoted to identifying
and analyzing such a choice.  We do not know how to prove that these new
simplices exist in all dimensions, but we prove existence in many hitherto
unknown cases. We also extend our methods to handle some exceptional cases
that are particularly subtle.

Our results settle several open problems dating back to the early 1980s,
while raising new questions.
\begin{enumerate}
\item We show the existence of a $15$-point simplex in $\HH \Proj^2$
    (Theorem~\ref{thm:15}) and a $27$-point simplex in $\OO \Proj^2$
    (Theorem~\ref{thm:27}).  These simplices are not only optimal codes,
    but also the largest possible simplices in their ambient spaces. For
    comparison, the six diagonals of an icosahedron form the largest
    simplex in $\R\Proj^2$, and the largest simplex in $\C\Proj^2$ has size
    nine. The real and complex simplices have long been known, but the
    quaternionic and octonionic simplices had been conjectured not to exist
    \cite[p.~251]{Ho1}.
\item These two codes are also tight $2$-designs, which makes them
    analogues of ``symmetric, informationally complete, positive
    operator-valued measures'' (SIC-POVMs).  SIC-POVMs are $d^2$-point
    simplices in $\C \Proj^{d-1}$, which play an important role in quantum
    information theory \cite{SICPOVM}.  Zauner has conjectured that
    SIC-POVMs exists for each $d$ (a problem analogous in some ways to the
    existence of Hadamard matrices), but his conjecture remains
    tantalizingly out of reach \cite{AFZ}; examples of SIC-POVMs have been
    algebraically constructed in low dimensions (up to $d=16$ and a few
    larger cases) and numerically approximated for $d \le 67$, but no
    infinite families are known \cite{SG,ABBEGL}. Our results do not apply
    directly to Zauner's conjecture, but rather they suggest that the
    analogous question in quaternionic projective spaces has an entirely
    different character: tight $2$-designs in $\HH \Proj^{d-1}$ do not seem
    to exist for $d > 3$.
\item An intriguing phenomenon we have observed is the apparent
    nonexistence of $14$-point tight simplices in $\HH \Proj^2$
    (Conjecture~\ref{conj:no14}) and $26$-point tight simplices in $\OO
    \Proj^2$ (Conjecture~\ref{conj:no26}), when all other sizes up to the
    maximum (15 or 27, respectively) occur.  We have no explanation for why
    the second largest possible size should not occur.  We observe the same
    phenomenon for the Grassmannian $G(2,4)$, which bolsters our confidence
    that it is a genuine pattern.
\item More generally, we prove the existence of many tight simplices in
    real Grassmannians (Theorems~\ref{thm:grass}
    and~\ref{thm:11inG25}--\ref{thm:7and8}). Such simplices were
    conjectured to exist in \cite{CHS} based on numerical evidence, and we
    show how parameter counting explains this phenomenon. As in projective
    spaces, the difficulty lies in finding the right constraints, so that
    the problem becomes amenable to rigorous proof.
\item Finally, we give a few explicit constructions of other codes,
    including a construction for a set of $13$ mutually unbiased bases in
    $\OO \Proj^2$ (Theorem~\ref{thm:39}). They had been conjectured to
    exist \cite[p.~35]{Ho2}, but no construction was previously known.
\end{enumerate}

In contrast to the usual algebraic methods for constructing tight codes, we
take a rather different approach to show the existence of families of
simplices. We use a general \emph{effective} implicit function theorem (i.e.,
one with explicit bounds), which allows us to show the existence of a real
solution to a system of polynomial equations near an approximate solution.
Furthermore, it proves that the space of solutions is a smooth manifold near
the approximate solution and tells us its dimension.  Using this approach, we
prove the existence of tight simplices by computing numerical approximations
and then applying the existence theorem.\footnote{In real projective spaces,
the problem is much easier: one can easily convert an approximation to an
exact construction by rounding the Gram matrix.  However, that fails in other
projective spaces and Grassmannians.  See the discussion before
Proposition~\ref{prp:tightoptimal}.}

The idea of making the implicit function theorem effective goes back to the
Newton-Kantorovich theorem \cite{Kantorovich}, but applying it in this
geometric setting allows us to establish many new results, for which
algebraic constructions seem out of reach. The closest predecessor to our
applications that we are aware of is a sequence of papers
\cite{ChenWomersley, Chen, AnChenSloanWomersley, ChenFrommerLang} on the
existence of spherical $t$-designs on $S^2$ with at least $(t+1)^2$ points.
These papers also use a Newton-Kantorovich variant, applied in a case in
which there are approximately twice as many variables as constraints: the space of $N$-point
configurations on $S^2$ has dimension $2N-3$ for $N \ge 3$, and the $t$-design
condition imposes $(t+1)^2-1$ constraints (since that is the
dimension of the space spanned by the spherical harmonics of degree $1$
through $t$).

In \textsection\ref{known} we describe linear programming bounds and recall
what is known about tight codes in projective spaces over $\R$, $\C$, $\HH$,
and $\OO$.  An effective existence theorem, our main tool in this paper, is
the subject of \textsection\ref{effective}. Our results concerning existence
of new families of projective simplices, proved using the existence theorem,
are described in \textsection\ref{quat} and \textsection\ref{oct}. In
\textsection\ref{grass} we use our methods to produce positive-dimensional
families of simplices in real Grassmannians.  We then give a discussion of
the algorithms and computer programs used for these computer-assisted proofs
in \textsection\ref{algs}. Finally, we conclude in \textsection\ref{mub} with
three explicit constructions of universally optimal codes, the most notable
of which is a maximal system of mutually unbiased bases in $\OO \Proj^2$.

We thank Noam Elkies for many helpful conversations.  We are especially
grateful to Mahdad Khatirinejad for his involvement in the early stages of
this work.  In particular, he collaborated with us on the numerical
investigations that initially suggested the widespread existence of tight
quaternionic simplices.  We are also grateful to the anonymous referees for
their suggestions.

\section{Codes in projective spaces and linear programming bounds} \label{known}

\subsection{Projective spaces over $\R$, $\C$, $\HH$, and $\OO$}

If $K = \R$, $\C$, or $\HH$, we denote by $K\Proj^{d-1} := (K^d \setminus
\{0\})/K^\times$ the set of lines in $K^d$. That is, we identify $x$ and
$x \alpha$ for $x \in K^d\setminus \{0\}$ and $\alpha \in K^\times$.  Note
the convention that $K^\times$ acts on the right; this is important for
the noncommutative algebra $\HH$.

We equip $K^d$ with the Hermitian inner product $\langle x_1, x_2 \rangle =
x_1^\dagger x_2$, where $\dagger$ denotes the conjugate transpose. We may
represent an element of the projective space $K\Proj^{d-1}$ by a unit-length vector $x \in
K^d$, and we often abuse notation by treating the element itself as such a
vector. Under this identification, the \emph{chordal distance} between two
points of $K\Proj^{d-1}$ is
\[
\rho(x_1,x_2) = \sqrt{1 - |\langle x_1,x_2 \rangle|^2}.
\]
It is not difficult to check that this formula defines a metric
equivalent to the Fubini-Study metric. Specifically, if $\vartheta(x_1,x_2)$
is the geodesic distance on $K\Proj^{d-1}$ under the Fubini-Study metric,
normalized so that the greatest distance between two points is $\pi$, then
\[
\cos \vartheta(x_1,x_2) = 2 |\langle x_1,x_2 \rangle |^2 - 1
\]
and
\[
\rho(x_1,x_2) = \sin \left( \frac{\vartheta(x_1,x_2)}{2} \right).
\]

Alternatively, elements $x \in K\Proj^{d-1}$ correspond to projection
matrices $\Pi = x x^\dagger$, which are Hermitian matrices with $\Pi^2 = \Pi$
and $\Tr \Pi = 1$.  The space $\mathcal{H}(K^d)$ of Hermitian matrices is a
real vector space endowed with a positive-definite inner product
\[
\langle A , B \rangle = \Tr \frac{1}{2}(AB+BA) = \RealPart \Tr AB.
\]
Because $\RealPart ab = \RealPart ba$ for $a,b \in K$, it follows that
$\RealPart \Tr (ABC) = \RealPart \Tr(CAB)$ for $A,B,C \in K^{d \times
  d}$; in other words, the functional $\RealPart \Tr$ is \emph{cyclic
  invariant}. Hence, for any $x_1,x_2 \in K\Proj^{d-1}$ with
associated projection matrices $\Pi_1,\Pi_2 \in \mathcal{H}(K^d)$, we
have
\begin{equation}
\label{eq:piproduct}
\begin{split}
\langle \Pi_1 , \Pi_2 \rangle &= \RealPart \Tr x_1 x_1^\dagger x_2 x_2^\dagger\\
&= \RealPart \Tr x_2^\dagger x_1 x_1^\dagger x_2 \\
&= \RealPart{} \langle x_2 , x_1 \rangle \langle x_1 , x_2 \rangle\\ % {} avoids bad spacing
&= |\langle x_1,x_2 \rangle|^2.
\end{split}
\end{equation}
Thus the metric on $K\Proj^{d-1}$ can also be defined by $\rho(x_1,x_2) =
\sqrt{1 - \langle \Pi_1,\Pi_2 \rangle}$.  Equivalently, it equals
$||\Pi_1-\Pi_2||_F/\sqrt{2}$, where $||\cdot||_F$ denotes the Frobenius norm:
$||A||_F = \big(\sum_{i,j} |A_{ij}|^2 \big)^{1/2}$ for a matrix whose $i,j$
entry is $A_{ij}$.

Modulo isometries, distance is the only invariant of a pair of
points, but triples have another invariant, known as the
\emph{Bargmann invariant} \cite{B64} or \emph{shape invariant}
\cite{B90}. In terms of projection matrices, it equals
$\RealPart \Tr {}\big(\Pi_1 \Pi_2 \Pi_3\big)$, and the
information it conveys is essentially the symplectic area of
the corresponding geodesic triangle \cite{MS93,HM94}. One can
define similar invariants for more than three points, but they
can be computed in terms of three-point invariants as long as
no two points are orthogonal.  When no two points are
orthogonal, the two- and three-point invariants characterize
the entire configuration \cite{BE98,BE01}.

The one remaining projective space we have not yet constructed
is the octonionic projective plane $\OO\Proj^2$. (See \cite{Ba}
for an account of why $\OO\Proj^d$ cannot exist for $d > 2$;
one can construct $\OO\Proj^1$, but we will ignore it as it is
simply $S^8$.) Due to the failure of associativity, the
construction of $\OO\Proj^2$ is more complicated than that of
the other projective spaces; in particular, we cannot simply view it as the
space of lines in $\OO^3$. However, there is a construction
analogous to the one using Hermitian matrices above. The result
is an exceptionally beautiful space that has been called the
panda of geometry \cite[p.~155]{Be}. The points of $\OO\Proj^2$
are $3 \times 3$ projection matrices over $\OO$, i.e., $3
\times 3$ Hermitian matrices $\Pi$ satisfying $\Pi^2 = \Pi$ and
$\Tr \Pi = 1$.  The (chordal) metric in $\OO\Proj^2$ is given
by
\[
\rho(\Pi_1,\Pi_2) = \frac{1}{\sqrt{2}}||\Pi_1 - \Pi_2||_F = \sqrt{1 -  \langle \Pi_1,\Pi_2 \rangle}.
\]
Each projection matrix $\Pi$ is of the form
\[
\Pi = \begin{pmatrix} a \\ b \\ c \end{pmatrix} \begin{pmatrix} \bar a & \bar b & \bar c \end{pmatrix},
\]
where $a,b,c \in \OO$ satisfy $|a|^2+|b|^2+|c|^2 = 1$ and $(ab)c = a(bc)$. We
can cover $\OO\Proj^2$ by three affine charts as follows. Any point may be
represented by a triple $(a,b,c) \in \OO^3$ with $|a|^2+|b|^2+|c|^2=1$, and
for the three charts we assume $a$, $b$, or $c$ are in $\R_+$, respectively.
In practice, for computations with generic configurations we can simply work
in the first chart and refer to a projection matrix by its associated point
$(a,b,c) \in \R_{+} \times \OO^2$.

\subsection{Tight simplices} \label{subsec:tight}

Projective spaces can be embedded into Euclidean space by mapping each point
to the corresponding projection matrix; using this embedding, the standard
bounds on the size and distance of regular Euclidean simplices imply
bounds on projective simplices. The resulting bounds, which we review in this subsection,
were first proven by Lemmens and Seidel
\cite{LeS}. They are also known in information theory as Welch bounds
\cite{W}.

As above, let $K$ be $\R$, $\C$, $\HH$, or $\OO$. We consider regular
simplices in $K\Proj^{d-1}$, with the understanding that $d = 3$ when $K =
\OO$.

\begin{definition}
A \emph{regular simplex} in a metric space $(X,\rho)$ is a collection of
distinct points $x_1, \dots, x_N$ of $X$ with the distances $\rho(x_i, x_j)$
all equal for $i \neq j$.
\end{definition}

We often drop the adjective ``regular'' and refer to a regular simplex as a simplex.

\begin{proposition} \label{prp:bounds} Consider a regular simplex in $K \Proj^{d-1}$
  consisting of $N>1$ points $x_1,\dots,x_N$
  with associated projection matrices $\Pi_1,\dots,\Pi_N$, and
  let $\alpha = \langle \Pi_i,\Pi_j \rangle$ be the common inner
  product for $i \ne j$.  Then
\[
N \le d + \frac{(d^2 - d)\dim_\R K}{2}
\]
and, for any such value of $N$,
\[
\alpha \ge \frac{N-d}{d(N-1)}.
\]
\end{proposition}

\begin{proof}
  The Gram matrix $G$ associated to
  $\Pi_1,\dots,\Pi_N$ has unit diagonal and $\alpha$ in each
  off-diagonal entry.  Since $G$ is
  nonsingular,\footnote{Specifically, $G = (1-\alpha)I_N + \alpha vv^t$, where $v$
  is the all-ones vector, and therefore the eigenvalues of $G$ are $1-\alpha$ (with multiplicity $N-1$)
  and $1+(N-1)\alpha$.
  These are all nonzero because $\alpha \in [0,1)$.} the elements
    $\Pi_1,\dots,\Pi_N \in \mathcal{H}(K^d)$ are linearly independent,
    implying $N \le \dim_{\R} \mathcal{H}(K^d) = d + (d^2 -
    d)(\dim_\R K)/2$.  Now note that $\langle \Pi_i , I_d \rangle = |x_i|^2 = 1$
    for each $i=1,\dots,N$.  Using this we compute
\[
\left \langle \left ( \sum_{i=1}^N \Pi_i \right ) - \frac{N}{d} I_d ,
  \left ( \sum_{i=1}^N \Pi_i \right ) - \frac{N}{d} I_d \right \rangle =
N - \frac{N^2}{d} + N(N-1) \alpha.
\]
Nonnegativity of this expression gives the desired bound on $\alpha$.
\end{proof}

\begin{definition} \label{def:tightsimplex}
We refer to a regular simplex with
\[
\alpha = \frac{N-d}{d(N-1)}
\]
as a \emph{tight simplex}. That is, it is a simplex with the maximum possible
distance allowed by Proposition~\ref{prp:bounds}.
\end{definition}

We noted above the difference between tight codes and tight designs,
and on the surface Definition~\ref{def:tightsimplex} seems to introduce
a third notion of tightness.  However, we will see that a tight simplex is
a tight code (Lemma~\ref{lma:tightmeanstight}), so this new definition is
really just a specialization.

Note that Definition~\ref{def:tightsimplex} is independent of the coordinate algebra $K$.  In
other words, the canonical embeddings $\R\Proj^{d-1} \hookrightarrow \C\Proj^{d-1} \hookrightarrow
\HH\Proj^{d-1}$ and $\HH\Proj^2 \hookrightarrow \OO\Proj^2$ preserve tight simplices.

It is not known for which $N$, $d$, and $K$ a tight simplex exists (later in
this section we will survey the known examples).  When $K = \R$, this problem
is fundamentally combinatorial.  Specifically, consider the Gram matrix of
some corresponding unit vectors in $\R^d$.  All the off-diagonal entries must
be $\pm \sqrt{\tfrac{N-d}{d(N-1)}}$, and the simplex is determined by the
sign pattern.  Thus, up to isometry, there can be only finitely many tight simplices of a
given size in $\R\Proj^{d-1}$.  Furthermore, any sufficiently close numerical
approximation will determine the signs and let one reconstruct the exact
simplex.

By contrast, tight simplices are much more subtle when $K \ne
\R$. The Gram matrix entries have phases, not just signs, and
tight simplices can even occur in positive-dimensional
families. In terms of the Bargmann invariants, the three-point
invariants are not determined by the pairwise distances. No
simple way to reconstruct an exact simplex from an
approximation is known, and we see no reason to believe one
exists.

\begin{proposition} \label{prp:tightoptimal}
Every tight simplex is an optimal code.
\end{proposition}

More generally, the bound on $\alpha$ in Proposition~\ref{prp:bounds} applies
to the minimal distance of any code, not just a simplex.

\begin{proof}
Let $\Pi_1,\dots,\Pi_N$ be the projection matrices corresponding to any
$N$-point code in $K\Proj^{d-1}$. As in the proof of
Proposition~\ref{prp:bounds},
\[
N - \frac{N^2}{d} + \sum_{\substack{i, j =1\\i \ne j}}^N \langle \Pi_i, \Pi_j\rangle
=
\left \langle \left ( \sum_{i=1}^N \Pi_i \right ) - \frac{N}{d} I_d ,
  \left ( \sum_{i=1}^N \Pi_i \right ) - \frac{N}{d} I_d \right \rangle
  \ge 0.
\]
Thus, the average of $\langle \Pi_i, \Pi_j \rangle$ over all $i \ne j$
satisfies
\[
\frac{1}{N(N-1)} \sum_{\substack{i, j =1\\i \ne j}}^N \langle \Pi_i, \Pi_j\rangle
\ge \frac{N^2/d-N}{N(N-1)} = \frac{N-d}{d(N-1)}.
\]
In particular, the greatest value of $\langle \Pi_i,\Pi_j \rangle$ for $i \ne
j$ must be at least this large.
\end{proof}

A regular simplex of $N \le d$ points in $K\Proj^{d-1}$ is
optimal if and only if the points are orthogonal (i.e.,
$\alpha=0$). Such simplices always exist. We only consider them
to be tight when $N=d$, as the $N < d$ cases do not satisfy
Definition~\ref{def:tightsimplex}; these degenerate cases are
tight simplices in a lower-dimensional projective space. There
also always exists a tight simplex with $N=d+1$ points,
obtained by projecting the regular simplex on the sphere
$S^{d-1}$ into $\R\Proj^{d-1}$. Therefore in what follows we
will generally assume $N \ge d+2$.

It follows immediately from the proof of Proposition~\ref{prp:bounds} that a
regular simplex $\{x_1,\dots,x_N\}$ is tight if and only if
\[
\sum_{i=1}^N x_i x_i^\dagger = \frac{N}{d} I_d.
\]
This condition can be reformulated in the language of projective designs
\cite{DGS,N81} (see also \cite{Ho1} for a detailed account of the relevant
computations in projective space).  Specifically, it says that the
configuration is a $1$-design.  We will make no serious use of the theory of
designs in this paper, and for our purposes we could simply regard
$\sum_{i=1}^N x_i x_i^\dagger = (N/d) I_d$ as the definition of a $1$-design.
However, to put our discussion in context, we will briefly recall the general concept
of designs in the next subsection.

\subsection{Linear programming bounds}

Linear programming bounds \cite{KL,DGS} use harmonic analysis on a space $X$
to prove bounds on codes in $X$.  These bounds and their extensions \cite{BV}
are among the only known ways to prove systematic bounds on codes, and they
are sharp in a number of important cases.  Later in this section we
summarize the sharp cases that are known in projective spaces (see also
Table~1 in \cite{CK} for a corresponding list for spheres), but first we
give a brief review of how linear programming bounds work.

The simplest setting for linear programming bounds is a compact two-point
homogeneous space.  We will focus on the connected examples, namely spheres
and projective spaces, but discrete two-point homogeneous spaces such as the
Hamming cube are also important in coding theory.

Let $X$ be a sphere or projective space, and let $G$ be its
isometry group under the geodesic metric $\vartheta$
(normalized so that the greatest distance is $\pi$).  Then
$L^2(X)$ is a unitary representation of $G$, and we can
decompose it as a completed direct sum
\[
L^2(X) = \hilbertsum\limits_{k \ge 0} V_k
\]
of irreducible representations $V_k$.  There is a corresponding sequence of
\emph{zonal spherical functions} $C_0, C_1, \dots$, one attached to each
representation $V_k$.  The zonal spherical functions are most easily obtained
as reproducing kernels; for a brief review of the theory, see Sections~2.2
and~8 of \cite{CK}.  We can represent them as orthogonal polynomials with
respect to a measure on $[-1,1]$, which depends on the space $X$, and we
index the polynomials so that $C_k$ has degree $k$.

For our purposes, the most important property of zonal spherical functions is
that they are \emph{positive-definite kernels}: for all $N \in \N$ and
$x_1,\dots,x_N \in X$, the $N \times N$ matrix $\big(C_i(\cos
\vartheta(x_i,x_j))\big)_{1 \le i,j \le N}$ is positive semidefinite. In
fact, the zonal spherical functions span the cone of all such functions.

For projective spaces $K \Proj^{d-1}$, the polynomials $C_k$
may be taken to be the Jacobi polynomials
$P_k^{(\alpha,\beta)}$, where $\alpha = (d-1)(\dim_\R K)/2-1$
(i.e, $\alpha = (\dim_\R K\Proj^{d-1})/2-1$) and $\beta =
(\dim_\R K)/2 - 1$. We will normalize $C_0$ to be $1$.

Linear programming bounds for codes amount to the following proposition:

\begin{proposition} \label{prop:LP}
Let $\theta \in [0,\pi]$, and suppose the polynomial
\[
f(z) = \sum_{k=0}^n f_k C_k(z)
\]
satisfies $f_0>0$, $f_k \ge 0$ for $1 \le k \le n$, and $f(z) \le 0$ for $-1
\le z \le \cos \theta$.  Then every code in $X$ with minimal geodesic
distance at least $\theta$ has size at most $f(1)/f_0$.
\end{proposition}

\begin{proof}
Let $\sC$ be such a code.  Then
\[
\sum_{x,y \in \sC} f(\cos \vartheta(x,y)) \ge f_0 |\sC|^2,
\]
because each zonal spherical function $C_k$ is positive definite and hence
satisfies
\[
\sum_{x,y \in \sC} C_k(\cos \vartheta(x,y)) \ge 0.
\]
On the other hand, $f(\cos \vartheta(x,y)) \le 0$ whenever $\vartheta(x,y)
\ge \theta$, and hence
\[
\sum_{x,y \in \sC} f(\cos \vartheta(x,y)) \le |\sC| f(1)
\]
because only the diagonal terms contribute positively. It follows that $f_0
|\sC|^2 \le f(1) |\sC|$, as desired.
\end{proof}

We say this bound is \emph{sharp} if there is a code $\sC$ with minimal
distance at least $\theta$ and $|\sC| = f(1)/f_0$.  Note that we require
exact equality, rather than just $|\sC| = \lfloor f(1)/f_0 \rfloor$.

\begin{definition}
A \emph{tight code} is one for which linear programming bounds are sharp.
\end{definition}

Examining the proof of Proposition~\ref{prop:LP} yields the following
characterization of tight codes:

\begin{lemma} \label{lemma:tightchar}
A code $\sC$ with minimal geodesic distance $\theta$ is tight iff there is a
polynomial $f(z) = \sum_{k=0}^n f_k C_k(z)$ satisfying $f_0>0$, $f_k \ge 0$
for $1 \le k \le n$, $f(z) \le 0$ for $-1 \le z \le \cos \theta$,
\[
\sum_{x,y \in \sC} C_k(\cos \vartheta(x,y)) = 0
\]
whenever $f_k>0$ and $k \ne 0$, and $f(\cos \vartheta(x,y))=0$
for $x,y \in \sC$ with $x \ne y$.  In fact, these conditions
must hold for every polynomial $f$ satisfying both
$f(1)/f_0=|\sC|$ and the hypotheses of
Proposition~\ref{prop:LP}.
\end{lemma}

By Proposition~\ref{prop:LP}, every tight code is as large as possible given
its minimal distance, but it is less obvious that such a code maximizes
minimal distance given its size.

\begin{proposition}
Every tight code is optimal.
\end{proposition}

\begin{proof}
Suppose $f$ satisfies the hypotheses of Proposition~\ref{prop:LP}, and $\sC$
is a code of size $f(1)/f_0$ with minimal geodesic distance at least
$\theta$.  We wish to show that its minimal distance is exactly $\theta$.

By Lemma~\ref{lemma:tightchar},
\[
\sum_{x,y \in \sC} \big(f(\cos \vartheta(x,y)) - f_0\big) = 0
\]
and $f(\cos \vartheta(x,y))=0$ for $x,y \in \sC$ with $x \ne y$

Now suppose $\sC$ had minimal geodesic distance strictly greater than
$\theta$, and consider a small perturbation $\sC'$ of $\sC$.  It must satisfy
\[
\sum_{x,y \in \sC'} \big(f(\cos \vartheta(x,y)) - f_0\big) \ge 0,
\]
by positive definiteness.  On the other hand,
\[
\sum_{x,y \in \sC'} \big(f(\cos \vartheta(x,y)) - f_0\big) =
|\sC'|f(1) - |\sC'|^2 f_0 + \sum_{\substack{x,y \in \sC'\\x \ne y}} f(\cos \vartheta(x,y)).
\]
We have $|\sC'|f(1) - |\sC'|^2 f_0 = 0$ since $|\sC'|=|\sC|=f(1)/f_0$.
Thus,
\[
\sum_{\substack{x,y \in \sC'\\x \ne y}} f(\cos \vartheta(x,y)) \ge 0.
\]
If the perturbation is small enough, then the minimal distance of $\sC'$
remains greater than $\theta$ and hence $f(\cos \vartheta(x,y)) \le 0$ for
distinct $x,y \in \sC'$. In that case, we must have $f(\cos \vartheta(x,y)) =
0$ for distinct $x,y \in \sC'$. However, this fails for some perturbations,
for example if we move two points slightly closer together. It follows that
every code of size $f(1)/f_0$ and minimal geodesic distance at least $\theta$
has minimal distance exactly $\theta$, so these codes are all optimal.
\end{proof}

\begin{lemma} \label{lma:tightmeanstight}
Tight simplices in projective space are tight codes.
\end{lemma}

\begin{proof}
Up to scaling, the first-degree zonal spherical function $C_1$ on $K
\Proj^{d-1}$ is $z+(d-2)/d$.  Now let
\[
f(z) = 1 + \frac{(N-1)d}{2(d-1)}\left(z + \frac{d-2}{d}\right).
\]
It satisfies $f(z) \le 0$ for $z \in [-1,2\alpha-1]$, where
\[
\alpha = \frac{N-d}{d(N-1)},
\]
and $f(1)/f_0 = N$, as desired.
\end{proof}

Note that in this proof $C_1$ depends only on $d$, and not on
$K$. By contrast, higher-degree zonal spherical functions for
$K\Proj^{d-1}$ depend on both $d$ and $K$.

We do not know whether every tight $N$-point code in $K\Proj^{d-1}$
with
\[
d \le N \le d + (d^2-d)(\dim_\R K)/2
\]
is a tight simplex, although we know of no counterexample.  This
assertion would follow if the linear function $f(z)$ from the proof of
Lemma~\ref{lma:tightmeanstight} always gave the optimal
bound for this range of $N$, but it does not.  For example,
consider $5$-point codes in $\R\Proj^2$, i.e., $d=3$, $K=\R$, $N=5$.
If there were a tight simplex with these parameters, then it would have
common squared inner product $\alpha = 1/6$, and positive definiteness
of $C_k$ would require that $C_k(1) + 4C_k(2\cdot 1/6-1) \ge 0$ for all $k$.
However, $C_4(1) + 4C_4(-2/3)<0$.  This means that no tight simplex
exists, and, in terms of linear programming bounds,%
\footnote{In fact, linear programming bounds prove a bound of $0.16866\dots$
(a cubic irrational) for the maximal squared inner product
of any $5$-point code in $\R\Proj^2$.  This bound is not achieved
by any real code, so in particular there is no tight $5$-point code, simplex or otherwise.}
it means that
we can improve on the linear function
$f(z)$ by replacing it with $f(z) + \varepsilon \big(C_4(z) - C_4(-2/3)\big)$,
for small positive $\varepsilon$.

A \emph{$t$-design} in $X$ is a code $\sC \subset X$ such that for every $f
\in V_k$ with $0 < k \le t$,
\[
\sum_{x \in \sC} f(x) = 0.
\]
In other words, every element of $V_0 \oplus \dots \oplus V_t$ has the same
average over $\sC$ as over the entire space $X$.  (Note that all functions in
$V_k$ for $k>0$ have average zero, since they are orthogonal to the constant
functions in $V_0$.)  Using the reproducing kernel property, this can be
shown to be equivalent to
\[
\sum_{x,y \in \sC} C_k(\cos \vartheta(x,y)) = 0
\]
for $0 < k \le t$.

In $K\Proj^{d-1}$, one can check that
\[
\sum_{i=1}^N x_i x_i^\dagger = \frac{N}{d} I_d
\]
holds if and only if $\{x_1,\dots,x_N\}$ is a $1$-design.

A code is \emph{diametrical} in $X$ if it contains two points at maximal
distance in $X$, and it is an \emph{$m$-distance set} if exactly $m$
distances occur between distinct points.

\begin{definition} A \emph{tight design} is an $m$-distance set that is a
$(2m-\varepsilon)$-design, where $\varepsilon$ is $1$ if the set is
diametrical and $0$ otherwise.
\end{definition}

For example, an $N$-point tight simplex in $K\Proj^{d-1}$ with $N = d + (d^2
- d)(\dim_\R K)/2$ (the largest possible value of $N$) is a tight $2$-design.
See \cite{BH} for further examples.

Every tight $t$-design is the smallest possible $t$-design in its ambient
space.  This was first proved for spheres in \cite{DGS}; see Propositions~1.1
and~1.2 in \cite{BH85} for the general case.  The converse is false: the
smallest $t$-design is generally not tight.

A theorem of Levenshtein \cite{Le1} says that every
$m$-distance set that is a $(2m-1-\varepsilon)$-design is a
tight code, where as above $\varepsilon$ is $1$ if the set is
diametrical and $0$ otherwise.  For example, all tight designs
are tight codes.  In \cite{CK}, it was also shown that under
these conditions, $\sC$ is \emph{universally optimal} for
potential energy: it minimizes energy for every completely
monotonic function of squared chordal distance.  (See also
\cite{Cohn} for context.) This applies in particular to
simplices, so all tight simplices are universally optimal.

In fact every known tight code is universally optimal.
Moreover, except for the regular $600$-cell in $S^3$ and its
image in $\R\Proj^3$, they all satisfy the design condition
just mentioned.  For lack of a counterexample, we conjecture
that tight codes are always universally optimal.  (But see
\cite{CZ} for perspective on why the simplest reason why this
might hold fails.)

\subsection{Tight codes in $\R\Proj^{d-1}$}

We now describe what is known about tight codes in real projective spaces.
Table~\ref{tab:runiv} provides a summary of the current state of knowledge.
Note that in several lines in the table, existence of a code is conditional
on existence of a combinatorial object such as a conference matrix; we
provide further details in the text below. See also Table~1 in \cite{W09}, which
provides a list of all known tight simplices in $\R\Proj^{d-1}$ with $d \le
50$ and all possible cases that have not been resolved.

\begin{table}
\caption{Known universal optima of $N$ points in real
  projective spaces $\R\Proj^{d-1}$.
  The tight simplices are indicated by an asterisk in the third column and have
  maximal squared inner product $(N-d)/(d(N-1))$; for brevity
  we omit the Gale duals of the tight simplices. A star in the last column
  means the code may exist only for certain parameter settings.} \label{tab:runiv}
\begin{center}
\begin{tabular}{cccc}
\toprule
$d$ & $N$ & $\max |\langle x,y \rangle|^2$ &  Name/origin \\
\midrule
$d$ & $N \le d$ & $0$ & orthogonal points (tight when $N=d$)\\
$d$ & $d+1$ & $*$ & Euclidean simplex \\
$d$ & $2d$ & $*$ & symm.\ conf.\ matrix of order $2d$\ \ ($\star$)\\
$d$ & $d(d+2)/2$ & $1/d$ & $d/2+1$ mutually unbiased bases\ \ ($\star$)\\
$2$ & $N$ & $\cos^2(\pi/N)$ & regular polygon \\
$4$ & $60$ & $(\sqrt{5} - 1)/4$ & regular $600$-cell\\
$6$ & $16$ & $*$ & Clebsch \\
$6$ & $36$ & $1/4$ & $E_6$ root system \\
$7$ & $28$ & $*$ & equiangular lines \\
$7$ & $63$ & $1/4$ & $E_7$ root system \\
$8$ & $120$ & $1/4$ & $E_8$ root system \\
$23$ & $276$ & $*$ & equiangular lines \\
$23$ & $2300$ & $1/9$ & kissing configuration of next line\\
$24$ & $98280$ & $1/4$ & Leech lattice minimal vectors\\
$\tfrac{v(v-1)}{k(k-1)}$ & $v\big(1+\tfrac{v-1}{k-1}\big)$ & $*$ & Steiner construction \ \ ($\star$)\\
& & & strongly regular graph with parameters \\
$d$ & $N$ & $*$ &  $(N-1,k,(3k-N)/2,k/2)$, where\\
& & & $k = \tfrac{N}{2} - 1 + \big(1-\tfrac{N}{2d}\big)\sqrt{\tfrac{d(N-1)}{N-d}}$\ \ ($\star$)\\
\bottomrule
\end{tabular}
\end{center}
\end{table}

Euclidean simplices and orthogonal points give the simplest
infinite families of tight codes.

Another infinite family of tight simplices comes from
conference matrices \cite{vLS} (see \cite[p.~156]{CHS}): if a
symmetric conference matrix of order $2d$ exists, then there is
a tight simplex of size $2d$ in $\R^d$. In particular, we get a
tight simplex in $\R^d$ whenever $2d-1$ is a prime power
congruent to $1$ modulo $4$. One can also construct such codes
through the Weil representation of the group $G =
\PSL_2(\F_q)$. Note that the icosahedron arises as the special
case $q=5$, which is why it is not listed separately in
Table~\ref{tab:runiv}.

Levenshtein \cite{Le2} described a family of tight codes in
$\R\Proj^{d-1}$ for $d$ a power of $4$, based on a construction
using Kerdock codes; the regular $24$-cell is the special case
with $d=4$. These codes meet the orthoplex bound (Corollary~5.3
in \cite{CHS}) and give rise to $d/2+1$ mutually unbiased bases
in their dimensions.  Recall that two orthonormal bases
$v_1,\dots,v_d$ and $w_1,\dots,w_d$ are \emph{mutually
unbiased} if $|\langle v_i, w_j \rangle|^2 = 1/d$ for all $i$
and $j$.

A trivial systematic family of tight codes is formed by the diameters of the
regular polygons in the plane.  The next nine lines in Table~\ref{tab:runiv}
correspond to exceptional geometric structures.

The Steiner construction from \cite{FMT12} builds a tight simplex from a
$(2,k,v)$ Steiner system and a Hadamard matrix of order $1+(v-1)/(k-1)$. See
\cite{FMT12} for a discussion of the parameters that can be achieved using
different sorts of Steiner systems. (Note that Bondarenko's tight simplex
\cite{Bo} is a Steiner simplex with $(k,v)=(3,15)$.)  Steiner simplices can
be constructed as follows. Recall that a $(2,k,v)$ Steiner system is a set of
$v$ points with a collection of subsets of size $k$ called blocks, such that
every two distinct points belonging to a unique block. Then there must be $d$
blocks, and every point is in $r$ blocks, where
\[
d = \frac{v(v-1)}{k(k-1)} \qquad \textup{and} \qquad r = \frac{v-1}{k-1}.
\]
Consider the $d \times v$ incidence matrix $A$ for blocks and points, with
entries $0$ and $1$, and let $H$ be a Hadamard matrix of order $r+1$. For
each $j$ from $1$ to $v$, consider the $j$th column of $A$, and form a $d
\times (r+1)$ matrix $M_j$ whose $i$th row is a different row of $H$ for each
$i$ satisfying $A_{i,j} \ne 0$ and vanishes otherwise. Then it is not
difficult to check that the $v(1+r)$ columns of all these matrices $M_j$ form
a tight simplex in $\R\Proj^{d-1}$.

The last entry in the table is a reformulation of tight simplices in
$\R\Proj^{d-1}$ in terms of strongly regular graphs (see Theorem~5.2 in
\cite{W09}).  This sort of combinatorial description works only over the real
numbers.  When $d \le 50$, only three cases are known that are not
encompassed by other lines in the table: $(d,N) = (22, 176)$, $(36,64)$, and
$(43,344)$. See Table~1 in \cite{W09} for more information.

We also observe the phenomenon of Gale duality: tight simplices
of size $N$ in $K\Proj^{d-1}$ correspond to tight simplices of
size $N$ in $K\Proj^{N-d-1}$. For instance, the Gale dual of
the Clebsch configuration gives a tight simplex of $16$ points
in $\R\Proj^9$. See \textsection\ref{subsec:gale} for more
details.

\subsection{Tight codes in $\C\Proj^{d-1}$}

Table~\ref{tab:cuniv} lists the tight codes we are aware of in
complex projective spaces. For a detailed survey of tight
simplices, we refer the reader to Chapter~4 of \cite{Kh}.

\begin{table}
\caption{Known universal optima of $N$ points in complex
  projective spaces $\C\Proj^{d-1}$.
  The tight simplices are indicated by an asterisk in the third column and have
  maximal squared inner product $(N-d)/(d(N-1))$; for brevity
  we omit the Gale duals of the tight simplices as well as the tight simplices from $\R\Proj^{d-1}$.
  A star in the last column
  means the code may exist only for certain parameter settings.}
\label{tab:cuniv}
\begin{center}
\begin{tabular}{cccc}
\toprule
$d$ & $N$ & $\max |\langle x,y \rangle|^2$ & Name/origin \\
\midrule
$d$ & $2d$ & $*$ & skew-symm.\ conf.\ matrix of order $2d$ \ \ ($\star$) \\
$d$ & $d^2$ & $*$ & SIC-POVMs \ \ ($\star$)\\
$d$ & $d(d+1)$ & $1/d$ &  $d+1$ mutually unbiased bases \ \ ($\star$)\\
$2k-1$ & $4k-1$ & $*$ & skew-Hadamard matrix of order $4k$\ \ ($\star$)\\
$2k$ & $4k-1$ & $*$ & skew-Hadamard matrix of order $4k$\ \ ($\star$)\\
$4$ & $40$  & $1/3$ & Eisenstein structure on $E_8$  \\
$5$ & $45$  & $1/4$ & kissing configuration of next line  \\
$6$ & $126$ & $1/4$ & Eisenstein structure on $K_{12}$  \\
$28$ & $4060$ & $1/16$ & Rudvalis group  \\
$\tfrac{v(v-1)}{k(k-1)}$ & $v\big(1+\tfrac{v-1}{k-1}\big)$ & $*$ & Steiner construction\ \ ($\star$)\\
$|S|$ & $|G|$ & $*$ & difference set $S$ in abelian group $G$\ \ ($\star$) \\
\bottomrule
\end{tabular}
\end{center}
\end{table}

Here, we observe a few more infinite families. In particular, if a conference
matrix of order $2d$ exists, then there is a tight code of $2d$ lines in
$\C\Proj^{d-1}$ \cite[p.~66]{Za}. For prime powers $q \equiv 3 \pmod 4$, this
gives a construction of a tight $(q+1)$-point code in $\C\Proj^{(q-1)/2}$. As
mentioned before, such codes may also be constructed using the Weil
representation of $\PSL_2(\F_q)$. Another family of codes of $d(d+1)$ points
in $\C \Proj^{d-1}$, for $d$ an odd prime power, was constructed by
Levenshtein \cite{Le2} using dual BCH codes. These codes meet the orthoplex
bound and give rise to $d+1$ mutually unbiased bases in their dimensions.
They were rediscovered by Wootters and Fields \cite{WF89}, with an extension
to characteristic~$2$ and applications to physics. A third infinite family is
obtained from skew-Hadamard matrices (see \cite{Re} for a construction using
explicit families of skew-Hadamard matrices and Theorem~4.14 in \cite{Kh} for
the general case).

The most mysterious tight simplices are the awkwardly named
SIC-POVMs (symmetric, informationally complete, positive
operator-valued measures).  SIC-POVMs are simplices of size
$d^2$ in $\C\Proj^{d-1}$, i.e., simplices of the greatest size
allowed by Proposition~\ref{prp:bounds}.  These configurations
play an important role in quantum information theory, which
leads to their name.  Numerical experiments suggest they exist
in all dimensions, and that they can even be taken to be orbits
of the Weyl-Heisenberg group \cite{Za,SICPOVM}.  Exact
SIC-POVMs are known for $d \le 16$, as well as $d = 19$, $24$,
$28$, $35$, and $48$, while numerical approximations are known for all
$d \le 67$ (see \cite{SG} and \cite{ABBEGL}).

The Steiner construction can be carried out in $\C\Proj^{d-1}$ using a
complex Hadamard matrix instead of a real Hadamard matrix (see \cite{FMT12}).
Complex Hadamard matrices of every order exist, so the construction
applies whenever there is a $(2,k,v)$ Steiner system.

The last line of the table refers to a construction based on
difference sets \cite{XZG} (see also \cite{Koenig}). Let $G$ be
an abelian group of order $N$, $S$ a subset of $G$ of order
$d$, and $\lambda$ a natural number such that every nonzero
element of $G$ is a difference of exactly $\lambda$ pairs of
elements of $S$. It follows that $d(d-1) = \lambda (N-1)$, and
that the vectors
\[
v_\chi = \left( \chi(s) \right)_{s \in S}
\]
give rise to a tight simplex of $N$ points in $\Proj^{d-1}$ as $\chi$ ranges
over all characters of $G$. As particular cases of this construction, one can
obtain a tight simplex of $n^2 + n + 1$ points in $\C\Proj^n$, when there is
a projective plane of order $n$. A generalization of this example was given
in \cite{XZG}, using Singer difference sets, to produce $(q^{d+1}-1)/(q-1)$
points in $\C\Proj^{d-1}$, with $d = (q^d -1)/(q-1)$. Similarly, if $q$ is a
prime power congruent to $3$ modulo $4$, then the quadratic residues give a
difference set, yielding a tight simplex of $q$ points in
$\C\Proj^{(q-3)/2}$. As another example, there is a difference set of $6$
points in $\Z/31\Z$ (namely, $\{0,1,4,6,13,21\}$), which gives rise to a
tight simplex of $31$ points in $\C\Proj^5$.

\subsection{Tight codes in $\HH\Proj^{d-1}$ and $\OO \Proj^2$}

\begin{table}
\caption{Previously known universal optima of $N$ points in
quaternionic and octonionic projective spaces.  For brevity
we omit the tight simplices from $\R\Proj^{d-1}$ and
$\C\Proj^{d-1}$.    A star in the last column
  means the code may exist only for certain parameter settings.} \label{tab:qouniv}
\begin{center}
\begin{tabular}{cccc}
\toprule
Space & $N$ & $\max |\langle x,y \rangle|^2$ & Name/origin  \\
\midrule
$\HH \Proj^{d-1}$ & $d(2d+1)$ & $1/d$ & $2d+1$ mutually unbiased bases\ \ ($\star$)\\
$\HH \Proj^4$ & $165$ & $1/4$ & quaternionic reflection group \\
$\OO \Proj^2$ & $819$ & $1/2$ & generalized hexagon of order $(2,8)$ \\
\bottomrule
\end{tabular}
\end{center}
\end{table}

Relatively little is known about tight codes in quaternionic or
octonionic projective spaces, aside from the real and complex
tight simplices they automatically contain.  When $d$ is a
power of $4$, there is a construction of $2d+1$ mutually
unbiased bases in $\HH\Proj^{d-1}$ due to Kantor \cite{Ka}, and
two exceptional codes are known.

The $165$ points in $\HH\Proj^4$ from Table~\ref{tab:qouniv} are constructed
using a quaternionic reflection group (Example~9 in \cite{Ho1}). The
$819$-point universal optimum is a remarkable code in the octonionic
projective plane \cite{Coh}; see also \cite{EG96} for another construction.
It can be thought of informally as the $196560$ Leech lattice minimal vectors modulo the
action of the $240$ roots of $E_8$ (viewed as units in the integral
octonions), although this does not yield an actual construction: there is no
such action because the multiplication is not associative.

\subsection{Gale duality} \label{subsec:gale}

Gale duality is a fundamental symmetry of tight simplices. It goes by several
names in the literature, such as coherent duality, Naimark complements, and
the theory of eutactic stars.  We call it Gale duality because it is a metric
version of Gale duality from the theory of polytopes (see Chapter~5 of
\cite{T06} for the non-metric Gale transform).

Let $K$ be $\R$, $\C$, or $\HH$.  (Gale duality does
not apply to $\OO\Proj^2$.)

\begin{proposition}[Hadwiger \cite{Had}] \label{galeprop}
Let $v_1, \dots, v_N$ span a $d$-dimensional vector space $V$
over $K$, and suppose they have the same norm $|v_i|^2 = d/N$.
Then their images in $K\Proj^{d-1}$ form a $1$-design if and
only if there is an $N$-dimensional vector space $U$ containing
$V$ and an orthonormal basis $u_1, \dots, u_N$ of $U$ such that
$v_i$ is the orthogonal projection of $u_i$ to $V$.
\end{proposition}

\begin{proof}
Let $M$ be the $d \times N$ matrix whose $i$th column is $v_i$. The existence
of $U$ and $u_1,\dots,u_N$ is equivalent to that of an extension of $M$ to a
unitary matrix by adding $N-d$ rows, in which case $u_1,\dots,u_N$ are the
columns of the extended matrix. This extension is possible if and only if the
rows of $M$ are orthonormal vectors; in other words, it is equivalent to $M
M^\dagger = I_d$.

To analyze $M$, we can write it as $M = \sum_{i=1}^N v_i
e_i^\dagger$, where $e_1,\dots,e_N$ is the standard orthonormal
basis of $K^N$.  Then
\[
M M^\dagger = \sum_{i,j=1}^N v_i e_i^\dagger e_j v_j^\dagger = \sum_{i=1}^N v_i v_i^\dagger.
\]
Thus, the extension is possible if and only if
\[
\sum_{i=1}^N v_i v_i^\dagger = I_d.
\]
This equation is the condition for a projective $1$-design once we
rescale to account for the normalization $|v_i|^2 = d/N$.
\end{proof}

Under the $1$-design condition from Proposition~\ref{galeprop}, consider the
projections $w_i$ of the vectors $u_i$ to the orthogonal complement $V^\perp$
of $V$ in $U$. This code $\{w_1,\dots,w_N\}$ in $K\Proj^{N-d-1}$ is called
the \emph{Gale dual} of the code $\{v_1, \dots, v_N\}$ in $K\Proj^{d-1}$. The
construction from the proof shows that the Gale dual is well defined up to
unitary transformations of $V^\perp$. However, there is one technicality: the
$N$ points in $K\Proj^{N-d-1}$ need not be distinct in general, so the Gale
dual must be considered a multiset of points.  Aside from the need to allow
multisets, Gale duality is an involution on projective $1$-designs, defined
up to isometry.

Gale duality preserves tight simplices when $N>d+1$, and the
multiplicity issue does not arise:

\begin{corollary} \label{cor:galedual}
Let $K$ be $\R$, $\C$, or $\HH$.  For $N > d+1$, the Gale dual
of an $N$-point tight simplex in $K \Proj^{d-1}$ is an
$N$-point tight simplex in $K \Proj^{N-d-1}$.
\end{corollary}

\begin{proof}
Because the $1$-design property is preserved, we need only
check that the Gale dual is a simplex.  In the notation used
above, for $i \neq j$ we have
\[
0 = \langle u_i , u_j \rangle = \langle v_i, v_j \rangle + \langle w_i, w_j \rangle.
\]
Thus, $\langle w_i, w_j \rangle$ is constant for $i \ne j$
because $\langle v_i, v_j \rangle$ is.  The inequality $N>d+1$
merely rules out the degenerate case $K\Proj^0$.
\end{proof}

The inequality
\[
N \le d + \frac{(d^2-d)\dim_\R K}{2}
\]
from Proposition~\ref{prp:bounds} shows that tight simplices
cannot be too large.  Combining Gale duality with the same
inequality shows that they cannot be too small either (see
Theorem~2.30 in \cite{Za} and Corollary~2.19 in \cite{Kh}):

\begin{corollary} \label{cor:galebound}
Let $K$ be $\R$, $\C$, or $\HH$.  If there exists an $N$-point
tight simplex in $K\Proj^{d-1}$ with $N>d+1$, then
\[
N \ge d + \frac{1+\sqrt{1+8d/(\dim_\R K)}}{2}.
\]
\end{corollary}

\section{Effective existence theorems} \label{effective}

Our main tool is an effective implicit function theorem, which gives
conditions under which an approximate solution to a system of equations
necessarily leads to a nearby exact solution. Theorems of this sort date back
to the Newton-Kantorovich theorem \cite{Kantorovich} on the convergence of
Newton's method (see also \cite{O} for a short proof).  Our formulation is
closer to Krawczyk's version of Newton-Kantorovich \cite{Kraw}, but it differs
in that we focus on existence of solutions rather than convergence of numerical
algorithms.

The following theorem is a variant of Theorem~2 in \cite{N07}, and we adapt the proof
given there.  In the statement, $||\cdot||$ denotes the operator
norm, $Df(x)$ is the Jacobian of $f$ at $x$,
$B(x_0,\varepsilon)$ is the open ball around $x_0$ with radius
$\varepsilon$, and $\id_W$ is the identity operator on $W$.

\begin{theorem} \label{thm:implicit}
Let $V$ and $W$ be finite-dimensional normed vector spaces over $\R$, and
suppose that $f \colon B(x_0,\varepsilon) \to W$ is a $C^1$ function, where
$x_0 \in V$ and $\varepsilon>0$. Suppose also that $T \colon W \to V$ is a
linear operator such that
\begin{equation} \label{eq:implicit1}
||Df(x) \circ T - \id_{W}|| < 1 - \frac{||T|| \cdot |f(x_0)|}{\varepsilon}
\end{equation}
for all $x \in B(x_0,\varepsilon)$. Then there exists $x_{*}
\in B(x_0, \varepsilon)$ such that $f(x_{*}) = 0$. Moreover, in $B(x_0,\varepsilon)$,
the zero locus $f^{-1}(0)$ is a $C^1$ submanifold
of dimension $\dim V - \dim W$.
\end{theorem}

Of course, the submanifold is smooth if $f$ is $C^\infty$.

\begin{proof}
Consider the initial value problem
\begin{equation} \label{eq:ivp}
x'(t) = -T \big(Df(x(t)) \circ T\big)^{-1} f(x_0), \qquad x(0) = x_0,
\end{equation}
which is a rescaling of the differential equation for the
continuous analogue of Newton's method (see Section~3 of
\cite{N07}).  The motivation is that
\begin{align*}
\frac{d}{dt} [ f(x(t)) ] &= Df(x(t))(x'(t)) \\
&= - \big(Df(x(t)) \circ T\big) \big(Df(x(t)) \circ T\big)^{-1} f(x_0) \\
&= -f(x_0),
\end{align*}
and so $f(x(t)) = (1-t)f(x_0)$.  Thus, $x(1)$ should be a root of $f$, but of
course we must verify that the initial value problem has a solution defined on $[0,1]$.

First note that the bound \eqref{eq:implicit1} implies that
$Df(x) \circ T$ is invertible for all $x \in B(x_0,\varepsilon)$.
Moreover, supposing for the moment that $f(x_0) \ne 0$, we have
\begin{equation} \label{eq:implicit-1}
||(Df(x) \circ T)^{-1}|| < \frac{\varepsilon}{||T|| \cdot |f(x_0)|}.
\end{equation}
These claims follow from the series expansion
\[
(Df(x) \circ T)^{-1} = \sum_{i=0}^\infty \big(\id_W-Df(x) \circ T \big)^i.
\]
Because $f$ is $C^1$, $(Df(x) \circ T)^{-1}$ is continuous.  Thus,
by the Peano existence theorem (see Chapter~1, Sections~1--5 in \cite{CL55}), the initial value problem
\eqref{eq:ivp} has a $C^1$ solution $x(t)$ defined on a
nontrivial interval starting at $0$.  The solution can be extended
as long as $x(t)$ does not approach the boundary
of $B(x_0,\varepsilon)$.
Using \eqref{eq:implicit-1}, we have
\[|x'(t)| \le ||T|| \cdot ||(Df(x(t)) \circ T)^{-1}|| \cdot |f(x_0)| < \varepsilon.\]
It follows that the solution
$x(t)$ can be continued to $t=1$ and satisfies $|x(t) - x_0| < \varepsilon t$;
setting $x_* = x(1)$ finishes the
first part of the theorem.

Of course, if $f(x_0) = 0$, then we can just take $x_* = x_0$.

It remains only to show that $f^{-1}(0)$ is a manifold of dimension
$\dim V-\dim W$.  We noted above that the operator $Df(x) \circ T$ is
invertible for all $x \in B(x_0,\varepsilon)$, so in particular this is true for
all $x \in f^{-1}(0)$.  But that implies that $Df(x)$ is surjective, so we are
done by an application of the standard
implicit function theorem (see Section~4.3 in \cite{KP13}).
\end{proof}

Given a function $f$ and an approximate root $x_0$, it is
straightforward to apply this theorem.  We must compute an
approximate right inverse $T$ of $Df(x_0)$ and bound $||Df(x)
\circ T - \id_W||$ for all $x \in B(x_0,\varepsilon)$.  The
simplest and most elegant way to do this is using interval
arithmetic (see \textsection\ref{algs} for details), but we can
also use Corollary~\ref{cor:polybounds} below when $f$ is a
polynomial.

In order for Theorem~\ref{thm:implicit} to prove the existence of a solution
of $f(x)=0$, $Df(x)$ must have a right inverse at that solution. (In
particular, we must have $\dim V \ge \dim W$.)  If we view $f$ as defining a
system of simultaneous equations, then choosing the right equations to use
can be tricky.  For example, some of the most straightforward systems
defining a tight simplex will not work to prove existence of such a simplex,
because $Df$ is singular at every solution.  Much of this paper is devoted to
formulating suitable systems defining different sorts of tight simplices. The
generic cases are reasonably straightforward, but even they must be handled
carefully, and a few extreme cases are particularly subtle
(Propositions~\ref{prp:15} and~\ref{prp:27}).

In our applications, $f$ will always be a polynomial map.
In this case, the following lemma can be useful in
conjunction with Theorem \ref{thm:implicit}.

\begin{definition}
For a polynomial $p \colon \R^m \to \R$ given by $p(x) = \sum_I
c_I x^I$, define $|p| = \sum_I |c_I|$.  Given a polynomial map
$p = (p_1,\dots,p_n) \colon \R^m \to \R^n$, define $|p| = \max
|p_i|$.
\end{definition}

\begin{lemma} \label{lma:polybounds}
Let $m \geq n$, $\varepsilon>0$, and $x_0 \in \R^m$.  Suppose $f \colon \R^m
\to \R^n$ is a polynomial function of total degree $d$, and let $\R^m$ and
$\R^n$ carry the $\ell_\infty$ norm. Set $\eta = \max(1, |x_0| + \varepsilon)$. Then
for all $x \in B(x_0, \varepsilon)$,
\[
||Df(x)  - Df(x_0)|| < |f| d(d-1) \varepsilon \eta^{d-2}.
\]
\end{lemma}

\begin{proof}
The $\ell_\infty \to \ell_\infty$ operator norm of a matrix is
the maximum of the $\ell_1$ norms of its rows, so we need to
bound the $\ell_1$ norm of each row of $Df(x) - Df(x_0)$.
Without loss of generality suppose $n=1$; in other words, work
with a fixed row of the matrix.  The quantity we want to bound
is
\[
A = \sum_{i=1}^m |\partial_i f(x) - \partial_i f(x_0)|,
\]
where $\partial_i f$ denotes the partial derivative of $f$ with respect to
the $i$th coordinate. Splitting this as a sum over the monomials of $f$, it
suffices, by the triangle inequality, to prove that $A < e(e-1) \varepsilon
\eta^{e-2}$ when $f$ is a (monic) monomial of total degree $e \le d$.  Using
the mean value theorem applied to the function $g(t) = \partial_i f (x_0 +
t(x - x_0))$, we have
\[
\partial_i f (x) - \partial_i f (x_0) = \sum_{j=1}^m \partial^2_{ij} f
(v_i) (x - x_0)_j
\]
for some $v_i$ on the line segment between $x_0$ and $x$ (where
$(x-x_0)_j$ denotes the $j$th coordinate of the vector
$x-x_0$). Therefore,
\[
A \le \sum_{i=1}^m \left | \sum_{j=1}^m \partial_{ij}^2 f(v_i)
  (x-x_0)_j \right | < \varepsilon \sum_{i,j=1}^m \left| \partial_{ij}^2 f (v_i)
\right|
\]
since the $\ell_\infty$ norm $|x - x_0|$ is bounded by $\varepsilon$. Write $f =
\prod_{k=1}^m x_k^{e_k}$.  Then $\partial_{ij}^2 f(v_i)$ equals a monomial of
degree $e-2$ times either $e_i e_j$ if $i \ne j$, or $e_i(e_i-1)$ if $i=j$.
Because $\eta \ge \max(|v_i|,1)$, the monomial is bounded by $\eta^{e-2}$.
Summing, we obtain
\[
A < \varepsilon \eta^{e-2} \left ( \sum_{i,j=1}^m e_i e_j - \sum_{i=1}^m e_i \right )
= \varepsilon \eta^{e-2} e(e-1),
\]
as desired.
\end{proof}

\begin{corollary} \label{cor:polybounds}
With notation as in Lemma~\ref{lma:polybounds}, if there exists a linear
operator $T \colon \R^n \to \R^m$ such that
\[
||Df(x_0) \circ T - \id_{\R^n} || + \varepsilon \, |f| d (d-1) \eta^{d-2}
||T|| < 1 - \frac{||T|| \cdot |f(x_0)|}{\varepsilon},
\]
then there exists $x_* \in B(x_0, \varepsilon)$ such that $f(x_*) = 0$, and
the zero locus $f^{-1}(0)$ is locally a manifold of dimension $m-n$.
\end{corollary}

\begin{proof}
Using $||Df(x) \circ T - \id_{\R^n} || \leq || Df(x) -
Df(x_0)|| \cdot || T || + ||Df(x_0) \circ T - \id_{\R^n}|| $,
we see that the hypotheses of Theorem \ref{thm:implicit} are
met.
\end{proof}

\section{Simplices in quaternionic projective spaces} \label{quat}

\subsection{Generic case}

The definition gives one characterization of tight $N$-point simplices; we
simply impose $|x_i|^2=1$ for each $i$ and $|\langle x_i,x_j \rangle|^2 =
(N-d)/(d(N-1))$ for $i < j$. In fact, tight simplices can be characterized
even more succinctly: it can be shown that $\sum_{i,j} |\langle x_i,x_j
\rangle|^2 \ge N^2/d$, with equality iff $\{x_1,\dots,x_N\}$ is a tight
simplex. Both of these descriptions, though, suffer from the problem that the
imposed conditions are \emph{singular}; loosely put, if a set of points
satisfies the conditions, then it does so ``just barely.'' For instance, if
we define $f \colon \HH^N \to \R^{N+1}$ by
\[
f( x_1, \dots, x_N) = \Big(|x_1|^2 - 1, \dots, |x_N|^2 - 1, \sum_{i,j}
|\langle x_i, x_j \rangle |^2 - N^2/d \Big),
\]
then the fact that the last coordinate is always nonnegative implies that the
last row of $Df$ is zero at a tight simplex. Therefore it is hopeless to try
to prove existence by applying Theorem~\ref{thm:implicit}. Setting all the
inner products equal to $(N-d)/(d(N-1))$ suffers from the same problem,
because
\[
\frac{1}{N(N-1)} \sum_{\substack{i, j =1\\i \ne j}}^N |\langle x_i,x_j \rangle|^2
\ge \frac{N-d}{d(N-1)}
\]
for all $x_1,\dots,x_N$ (see the proof of
Proposition~\ref{prp:tightoptimal}).

Fortunately, it is generally possible to recast the conditions
describing tight simplices so that the Jacobian of the associated
polynomial map becomes surjective.

\begin{proposition} \label{prp:general} Suppose $x_1,\dots,x_N \in
  \HH^{d}$ $(d > 1)$ and $w_1,\dots,w_N \in \R$ satisfy
  the following conditions:
  \begin{enumerate}
  \item $|x_i|^2 = 1$ for $i=1,\dots,N$,
  \item $|\langle x_i,x_j \rangle|^2 = |\langle
      x_{i'},x_{j'} \rangle|^2$ for $1 \le i < j \le N$ and
      $1 \le i' < j' \le N$, and
  \item $\sum_{i=1}^N w_i x_i x_i^\dagger = I_d$.
\end{enumerate}
Then $w_1 = \cdots = w_N = d/N$ and $\{x_1,\dots,x_N\}$ is a tight simplex in
$\HH\Proj^{d-1}$.
\end{proposition}

\begin{proof}
Define $\Pi_i = x_i x_i^\dagger$, and let $\alpha$ denote the
common inner product $|\langle x_i,x_j \rangle|^2$  for $i \ne
j$.  By the first condition we have $\langle \Pi_i , I_d
\rangle = 1$ for each $i$.  Thus
\[
d = \langle I_d,I_d \rangle = \sum_{i=1}^N w_i \langle \Pi_i ,
I_d \rangle = \sum_{i=1}^N w_i.
\]
Moreover, using equation \eqref{eq:piproduct} we have $\langle
\Pi_i , \Pi_i \rangle = 1$ and $\langle \Pi_i , \Pi_j \rangle =
\alpha$ for all $i \ne j$.  Thus, for any $j$,
\[1 = \langle \Pi_j , I_d \rangle = \sum_{i=1}^N w_i \langle \Pi_j,
\Pi_i \rangle = (1-\alpha) w_j + \alpha \cdot \sum_{i=1}^N w_i
= (1-\alpha) w_j + \alpha d.\] It follows that $w_j = (1 -
\alpha d)/(1-\alpha)$ for each $j$.  Substituting back into the
equation $\sum_{i=1}^N w_i = d$ yields $\alpha = (N-d) /
(d(N-1))$, from which the result follows.
\end{proof}

Using Proposition \ref{prp:general}, we can view tight simplices of $N$
points in $\HH\Proj^{d-1}$ as the solutions of a system of
\[
N + \left ( \frac{N(N-1)}{2} - 1 \right ) + (2d^2-d) \text{ real constraints}
\]
in
\[
N(4d+1)\text{ real variables}.
\]
In situations where Theorem~\ref{thm:implicit} applies to this system, we get
a solution space of dimension $(\text{number of variables}) - (\text{number
of constraints})$.  This separately counts each unit-norm lift of the $N$
elements of $\HH\Proj^{d-1}$, so the space of simplices has
codimension $3N$.  Moreover, the space of simplices is invariant under the
action of the symmetry group of $\HH\Proj^{d-1}$, and we are most interested in the
quotient, i.e., the moduli space of simplices.  This symmetry group, the
compact symplectic group $\operatorname{Sp}(d)$ (strictly speaking, modulo
its center $\{\pm1\}$), has real dimension $d(2d+1)$. Thus the actual
dimension of the moduli space of simplices, local to this particular solution,
is at least
\begin{equation}  \label{eq:generaldim}
\sdim{N}{\HH\Proj^{d-1}} := (4d-3)N - \frac{N(N-1)}{2} - 4d^2 + 1
\end{equation}
when Theorem~\ref{thm:implicit} and
Proposition~\ref{prp:general} apply. Equality holds if the simplices in this
neighborhood have finite stabilizers
(in which case the moduli space of simplices is locally an orbifold
of the desired dimension); in any case,
the moduli space always has dimension at least $\sdim{N}{\HH\Proj^{d-1}}$.

The discussion above is informal in the case of a positive-dimensional stabilizer,
but it is not difficult to make the lower
bound rigorous for topological dimension.
Specifically, the solution space $\X$ of the system of equations from
Proposition~\ref{prp:general}
is a compact metric space, and locally a manifold of dimension
$\sdim{N}{\HH\Proj^{d-1}} + 3N + \dim \operatorname{Sp}(d)$
near the solution we find.  Thus its topological dimension is at least that large.
The moduli space is $\X/G$, where $G =
\operatorname{Sp}(1)^N \times \operatorname{Sp}(d)$.  Because
$G$ is compact, the quotient map $\X \to \X/G$ is closed and $\X/G$ is
Hausdorff.  Thus, we can apply topological
dimension theory for separable metric spaces to conclude that
\[
\dim(\X/G) \ge \dim \X - \dim G \ge \sdim{N}{\HH\Proj^{d-1}},
\]
as desired (see Theorem~VI~7 from \cite[p.~91]{HW41}).

Note that Gale duality, which replaces $d$ with $N-d$, preserves
$\sdim{N}{\HH\Proj^{d-1}}$, as one would expect.  Furthermore,
because $\sdim{N}{\HH\Proj^{d-1}}$ is quadratic in $N$, it is
also symmetric about the midpoint of the range in which it is
positive.  Specifically, $\sdim{N}{\HH\Proj^{d-1}} =
\sdim{8d-5-N}{\HH\Proj^{d-1}}$.

While \emph{a priori} it is possible to have tight simplices of
up to $N=2d^2-d$ points, we only have $\sdim{N}{\HH\Proj^{d-1}}
\ge 0$ for $N$ between roughly $(4-2\sqrt{2})d$ and
$(4+2\sqrt{2})d$.  That does not rule out larger tight
simplices, but it does mean that this approach using
Proposition~\ref{prp:general} and Theorem~\ref{thm:implicit}
could not prove their existence.  We believe that outside of
this range, only sporadic examples will exist in general, but
we conjecture that tight simplices always exist within the
range where $\sdim{N}{\HH\Proj^{d-1}} \ge 0$, at least if one
stays away from the boundary:

\begin{conjecture}
As $d \to \infty$, there exist tight $N$-point simplices in
$\HH\Proj^{d-1}$ for all $N$ satisfying
\[
(4-2\sqrt{2}+o(1))d \le N \le (4+2\sqrt{2}-o(1))d.
\]
\end{conjecture}

\begin{remark} \label{rmk:whenrapplies}
We emphasize that $\sdim{N}{\HH\Proj^{d-1}}$ is \emph{defined} by
\eqref{eq:generaldim}.  The assertion that the moduli space of simplices locally has
dimension $\sdim{N}{\HH\Proj^{d-1}}$ is justified only when (i) we find a
numerical solution of the conditions of Proposition \ref{prp:general} to
which Theorem~\ref{thm:implicit} applies, and (ii) the action of the symmetry
group on our simplex has finite (zero-dimensional) stabilizer.  Regarding
(ii), we have checked this rigorously in all the cases in
part (a) of Tables~\ref{tbl:general3}--\ref{tbl:general6} (see
\textsection\ref{subsec:stabilizers}).  In Table~\ref{tbl:generalop2}, which
deals with $\OO\Proj^2$, only $5$-point simplices fail to satisfy condition
(ii). In that case there is a $3$-dimensional stabilizer. We accounted for
this in Table~\ref{tbl:generalop2}.
\end{remark}

\begin{remark}
Similar calculations based on the real and complex analogues of
Proposition~\ref{prp:general} yield
\[
\sdim{N}{\R\Proj^{d-1}} = dN - \frac{N(N-1)}{2} - d^2 + 1
\]
and
\[
\sdim{N}{\C\Proj^{d-1}} = (2d-1)N - \frac{N(N-1)}{2} - 2d^2 + 2.
\]
Neither quantity is ever positive when $d>2$, which explains why our methods
do not apply to real and complex projective spaces: the system of equations
cannot be nonsingular for any tight simplex whose stabilizer is
zero-dimensional.
\end{remark}

When we attempt to apply Proposition \ref{prp:general}, there are
three possible outcomes:
\begin{enumerate}
\item[(a)] we find an approximate numerical solution with
    surjective Jacobian, in which case we can prove
    existence using Theorem \ref{thm:implicit},
\item[(b)] we find an approximate numerical solution, but
    the Jacobian at that point is not surjective, or
\item[(c)] we cannot even find an approximate numerical
    solution to the system, in which case we conjecture
    that there exists no tight simplex.
\end{enumerate}
In a few cases we encountered a fourth possibility:
\begin{enumerate}
\item[(d)] we find what appears to be an approximate solution but we are
unable to converge to greater precision.
\end{enumerate}
When this situation arose we tried both Newton's method and
gradient descent for energy minimization (see
\textsection\ref{subsec:finding}), but we were unable to
improve the error in the constraints beyond $10^{-5}$ (as
compared to a numerical error of about $10^{-15}$ for cases (a)
and (b)). In these cases we make no conjecture as to existence
or nonexistence of solutions.

Tables \ref{tbl:general3}, \ref{tbl:general4},
\ref{tbl:general5}, and \ref{tbl:general6} list our results for
$d=3$, $d=4$, $d=5$, and $d=6$, respectively.  Each table lists
all values of $N$ from $d+2$ to the upper bound $2d^2-d$ from
Proposition~\ref{prp:bounds}. There is no intrinsic problem
with extending to larger dimensions, although the calculations
become increasingly time-consuming.

\begin{theorem}
For the values of $(N,d)$ listed in part (a) of
Tables~\ref{tbl:general3} through \ref{tbl:general6}, there
exist tight $N$-point simplices in $\HH\Proj^{d-1}$.
\end{theorem}

In fact, near the points found by our computer program and
exhibited in the auxiliary files, the moduli space of simplices
has dimension exactly $\sdim{N}{\HH\Proj^{d-1}}$.
In the case of a
singular Jacobian (part (b) of the tables) we report the
\emph{rank deficiency} (i.e., $\dim W - \mathop{\mathrm{rank}} Df(x_*)$ in the
terminology of Theorem~\ref{thm:implicit}).

\begin{table}
\caption{Cases in $\HH\Proj^2$: (a) proven existence of tight simplices;
(b) singular Jacobian; (c) conjectured nonexistence.}
\label{tbl:general3}

\centering
\begin{tabular}{>{\centering\arraybackslash\hspace{0pt}}p{0.3\textwidth}>{\centering\arraybackslash\hspace{0pt}}p{0.3\textwidth}>{\centering\arraybackslash\hspace{0pt}}p{0.3\textwidth}}
{\begin{tabular}[t]{cc}
\toprule
$N$ & {$\sdim{N}{\HH\Proj^2}$} \\
\midrule
5 & 0 \\
6 & 4 \\
7 & 7 \\
8 & 9 \\
9 & 10 \\
10 & 10 \\
11 & 9 \\
\bottomrule
\end{tabular}

\smallskip

(a)} &
{\begin{tabular}[t]{cc}
\toprule
$N$ & {rank deficiency} \\
\midrule
12 & 2 \\
13 & 2 \\
15 & 14 \\
\bottomrule
\end{tabular}

\smallskip

(b)} &
{\begin{tabular}[t]{c}
\toprule
$N$ \\
\midrule
14 \\
\bottomrule
\end{tabular}

\smallskip

(c)}
\end{tabular}
%\vskip 0.5cm % To separate from the next table.  Removed since tables not adjacent.
\end{table}

\begin{table}
\caption{Cases in $\HH\Proj^3$: (a) proven existence of tight simplices;
(c) conjectured nonexistence.}
\label{tbl:general4}

\centering
\begin{tabular}{>{\centering\arraybackslash\hspace{0pt}}p{0.5\textwidth}>{\centering\arraybackslash\hspace{0pt}}p{0.2\textwidth}}
{\begin{tabular}[t]{ccccc}
\toprule
$N$ & $\sdim{N}{\HH\Proj^3}$ & {\hspace{8pt}} & $N$ & $\sdim{N}{\HH\Proj^3}$ \\
\midrule
 6 & 0       & {}     &  14 & 28   \\
 7 & 7       & {}     &  15 & 27   \\
 8 & 13      & {}     &  16 & 25   \\
 9 & 18      & {}     &  17 & 22   \\
10 & 22      & {}     &  18 & 18   \\
11 & 25      & {}     &  19 & 13   \\
12 & 27      & {}     &  20 &  7   \\
13 & 28      & {}     &  21 &  0   \\
\bottomrule
\end{tabular}

\smallskip

(a)} &
{\begin{tabular}[t]{c}
\toprule
$N$ \\
\midrule
22--28 \\
\bottomrule
\end{tabular}

\smallskip

(c)}
\end{tabular}
\vskip 0.5cm % To separate from the next table.  Feel free to remove.
\end{table}

\begin{table}
\caption{Cases in $\HH\Proj^4$: (a) proven existence of tight
simplices; (c) conjectured nonexistence (proven for $N=7$); (d)
ambiguous numerical results.} \label{tbl:general5}

\centering
\begin{tabular}{>{\centering\arraybackslash\hspace{0pt}}p{0.75\textwidth}>{\centering\arraybackslash\hspace{0pt}}p{0.2\textwidth}}
{\begin{tabular}[t]{cccccccc}
\toprule
$N$ & $\sdim{N}{\HH\Proj^4}$ & {\hspace{8pt}} & $N$ & $\sdim{N}{\HH\Proj^4}$ & {\hspace{8pt}} & $N$ & $\sdim{N}{\HH\Proj^4}$ \\
\midrule
 8 &  9      & {}     &  15 & 51      & {}     &  22 & 44   \\
 9 & 18      & {}     &  16 & 53      & {}     &  23 & 39   \\
10 & 26      & {}     &  17 & 54      & {}     &  24 & 33   \\
11 & 33      & {}     &  18 & 54      & {}     &  25 & 26   \\
12 & 39      & {}     &  19 & 53      & {}     &  26 & 18   \\
13 & 44      & {}     &  20 & 51      & {}     &  27 &  9   \\
14 & 48      & {}     &  21 & 48      \\
\bottomrule
\end{tabular}

\smallskip

(a)} &
{\begin{tabular}[t]{c}
\toprule
$N$ \\
\midrule
7 \\
29--45 \\
\bottomrule
\end{tabular}

\smallskip

(c)

\bigskip

\begin{tabular}[t]{c}
\toprule
$N$ \\
\midrule
28 \\
\bottomrule
\end{tabular}

\smallskip

(d)}
\end{tabular}
\end{table}

\begin{table}
\caption{Cases in $\HH\Proj^5$: (a) proven existence of tight
simplices; (c) conjectured nonexistence (proven for $N=8$); (d)
ambiguous numerical results.} \label{tbl:general6}

\centering
\begin{tabular}{>{\centering\arraybackslash\hspace{0pt}}p{0.75\textwidth}>{\centering\arraybackslash\hspace{0pt}}p{0.2\textwidth}}
{\begin{tabular}[t]{cccccccc}
\toprule
$N$ & $\sdim{N}{\HH\Proj^5}$ & {\hspace{8pt}} & $N$ & $\sdim{N}{\HH\Proj^5}$ & {\hspace{8pt}} & $N$ & $\sdim{N}{\HH\Proj^5}$ \\
\midrule
 9 & 10      & {}     &  18 & 82      & {}     &  27 & 73   \\
10 & 22      & {}     &  19 & 85      & {}     &  28 & 67   \\
11 & 33      & {}     &  20 & 87      & {}     &  29 & 60   \\
12 & 43      & {}     &  21 & 88      & {}     &  30 & 52   \\
13 & 52      & {}     &  22 & 88      & {}     &  31 & 43   \\
14 & 60      & {}     &  23 & 87      & {}     &  32 & 33   \\
15 & 67      & {}     &  24 & 85      & {}     &  33 & 22   \\
16 & 73      & {}     &  25 & 82      & {}     &  34 & 10   \\
17 & 78      & {}     &  26 & 78      \\
\bottomrule
\end{tabular}

\smallskip

(a)} &
{\begin{tabular}[t]{c}
\toprule
$N$ \\
\midrule
 8 \\
36--66 \\
\bottomrule
\end{tabular}

\smallskip

(c)

\bigskip

\begin{tabular}[t]{c}
\toprule
$N$ \\
\midrule
35 \\
\bottomrule
\end{tabular}

\smallskip

(d)}
\end{tabular}
\end{table}

In Table \ref{tbl:general5}, i.e., in $\HH\Proj^4$, we first
observe a gap between the tight simplices of sizes $d$ and
$d+1$ that always exist in $\HH\Proj^{d-1}$ and the range of
simplices for which our method proves existence. The gap is
real: there exists no $7$-point tight simplex in $\HH\Proj^4$,
because of Corollary~\ref{cor:galebound}. Similarly, there
exists no $8$-point tight simplex in $\HH\Proj^5$.

\subsection{$12$- and $13$-point simplices}
\label{subsec:1213}

The cases of $12$- and $13$-point simplices are somewhat special: the
system of constraints specified by Proposition \ref{prp:general} has a
rank deficiency.  To prove existence of solutions using Theorem
\ref{thm:implicit}, a different approach is needed.

We take as our starting point the following observation: not only do tight
$12$-point simplices exist (numerically), but actually $12$-point
cyclic-symmetric simplices exist (again, numerically).  By this we mean a
simplex such that, if $(x,y,z) \in \HH^3$ is a point in it, then so are
$(y,z,x)$ and $(z,x,y)$, and these are three distinct points in $\HH\Proj^2$.

We would like to adapt Proposition \ref{prp:general} to find
simplices with cyclic symmetry.  Imposing this symmetry reduces
the number of degrees of freedom we have, but it also reduces
the number of conditions we need to check. Fortunately, we end
up with a set of constraints that has a surjective Jacobian at
a tight simplex.

For convenience we will state the result only for $d=3$, but it
naturally generalizes to any dimension (along the lines of Proposition~\ref{prp:11inG25}).

\begin{proposition} \label{prp:12} Let $\sigma$ be the cyclic-shift
  automorphism $\sigma(a,b,c) = (b,c,a)$.  Suppose $x_1,\dots,x_{3m}
  \in \HH^3$ and $w_1,\dots,w_{3m} \in \R$ satisfy the
  following conditions:
\begin{enumerate}
\item $x_{m+i} = \sigma(x_i)$ for $i=1,\dots,2m$,
\item $w_{m+i} = w_i$ for $i=1,\dots,2m$,
\item $|x_i|^2 = 1$ for $i=1,\dots,m$,
\item the squared inner products $|\langle x_i , x_j \rangle |^2$ for
    $i=1,\dots,m$ and the following values of $j$ are all equal:
    \caseup{(i)} $j=i+m$, \caseup{(ii)} $i < j \le m$, \caseup{(iii)}
    $i+m < j \le 2m$, \caseup{(iv)} $i+2m < j \le 3m$, and
\item the matrix $\sum_{i=1}^{3m} w_i x_i x_i^\dagger$ has $1,1$ entry
    equal to $1$ and vanishing $1,2$ entry.
\end{enumerate}
Then $w_1 = \cdots = w_{3m} = 1/m$ and $\{x_1,\dots,x_{3m}\}$ is a tight
simplex in $\HH\Proj^2$.
\end{proposition}

\begin{proof}
By repeated applications of $\langle x_i, x_j \rangle = \langle
\sigma(x_i), \sigma(x_j) \rangle$, it easily follows that
$\{x_1,\dots,x_{3m}\}$ is a simplex.

Having shown that, now consider the matrix $M = \sum_{i=1}^{3m}
w_i x_i x_i^\dagger$.  Rewriting $M$ as $\sum_{i=1}^m w_i (x_i
x_i^\dagger + \sigma(x_i) \sigma(x_i)^\dagger + \sigma^2(x_i)
\sigma^2(x_i)^\dagger)$, we see that $M$ is cyclic-symmetric;
in other words, it is invariant under conjugation by the permutation
$\sigma$.  Of course $M$ is also Hermitian.  Combining these
two properties, it must be of the form
\[
M = \begin{pmatrix}
  r & s & \bar s \\
  \bar s & r & s \\
  s & \bar s & r
\end{pmatrix}
\]
for some $r \in \R$ and $s \in \HH$.  The last condition in the proposition
statement forces $r = 1$ and $s = 0$, so in fact $M = I_3$.

Therefore, $\{x_1,\dots,x_{3m}\}$ is a simplex with
$\sum_{i=1}^{3m} w_i x_i x_i^\dagger = I_3$, and we complete
the proof by applying Proposition~\ref{prp:general}.
\end{proof}

Applying the constraints in the above proposition with $m=4$,
we get a surjective Jacobian in Theorem~\ref{thm:implicit},
which proves the following result.

\begin{theorem} \label{thm:12}
There is a tight simplex of $12$ points in $\HH\Proj^2$.  In fact, there is
such a tight simplex with cyclic symmetry.
\end{theorem}

Experimentally it appears that tight simplices with cyclic
symmetry exist in other cases (e.g., $6$- and $9$-point
simplices in $\HH\Proj^2$).  In those cases we do not need to
use the symmetry to establish the existence of tight simplices,
though.

For $13$-point simplices, we wish to follow a similar approach
to bypass the rank-deficiency issue, but we must allow fixed
points of the cyclic shift.  In fact, there are
cyclic-symmetric $13$-point tight simplices consisting of $12$
points with cyclic symmetry as above (i.e., four equivalence
classes under the cyclic-shift operator) plus one extra point
which is invariant under the cyclic-shift operator.

\begin{proposition} \label{prp:13} Let $\sigma$ be the cyclic-shift
  automorphism $\sigma(a,b,c) = (b,c,a)$.  Suppose $x_1,\dots,x_{3m}
  \in \HH^3$ satisfy the following conditions:
\begin{enumerate}
\item $x_{m+i} = \sigma(x_i)$ for $i=1,\dots,2m$,
\item $|x_i|^2 = 1$ for $i=1,\dots,m$,
\item the squared inner products $|\langle x_i , x_j
    \rangle |^2$ for $i=1,\dots,m$ and the following values
    of $j$ are all equal: \caseup{(i)} $j=i+m$,
    \caseup{(ii)} $i < j \le m$, \caseup{(iii)} $i+m < j
    \le 2m$, \caseup{(iv)} $i+2m < j \le 3m$, and
\item the $1,2$ entry of the matrix $\sum_{i=1}^{3m} x_i x_i^\dagger$
  has real part $1/6$ and magnitude $1/3$.
\end{enumerate}
Then there is a unique point $x_{3m+1} \in \HH\Proj^2$ such that
$\{x_1,\dots,x_{3m},x_{3m+1}\}$ is a tight simplex, and that point satisfies
$\sigma(x_{3m+1}) = x_{3m+1}$.
\end{proposition}

\begin{proof}
  A tight $(3m+1)$-point simplex $\{x_1,\dots,x_{3m+1}\}$ must satisfy
\[
\sum_{i=1}^{3m+1} x_i x_i^\dagger = \frac{3m+1}{3} I_3.
\]
Thus the
  matrix $x_{3m+1} x_{3m+1}^\dagger$ is determined by the other data;
  since a point in projective space is determined by its projection
  matrix, this proves uniqueness.  It also proves that, if such a
  point $x_{3m+1}$ exists, then it must satisfy $\sigma(x_{3m+1}) =
  x_{3m+1}$ up to scalar multiplication (by a cube root of unity in $\HH$);
  this is because otherwise
\[
\sigma(\{x_1,\dots,x_{3m},x_{3m+1}\}) =
\{x_1,\dots,x_{3m},\sigma(x_{3m+1})\}
\]
would be a distinct tight simplex.

Define $M = \sum_{i=1}^{3m} x_i x_i^\dagger$.  This matrix is
Hermitian and cyclic-symmetric, so as in the proof of Proposition
\ref{prp:12} it is of the form
\[
M = \begin{pmatrix}
  r & s & \bar s \\
  \bar s & r & s \\
  s & \bar s & r
\end{pmatrix}
\]
for some $r \in \R$ and $s \in \HH$.  Each projection $x_i x_i^\dagger$ has
trace $1$, so $\Tr M = 3m$ and thus $r = m$.  Let
\begin{align*}
\Pi &:= \frac{3m+1}{3} I_3 - M\\
& = \begin{pmatrix}
1/3 & -s & -\bar s \\
-\bar s & 1/3 & -s \\
-s & -\bar s & 1/3
\end{pmatrix}.
\end{align*}
Being Hermitian and of trace $1$, $\Pi$ is a projection matrix of rank $1$
iff $3 s^2 = -\bar s$, as one can see by solving $\Pi^2 = \Pi$. The last
hypothesis in the proposition statement implies that $-3s$ is a cube root of
unity in $\HH$, from which we see that this condition is satisfied.

Let $x_{3m+1} \in \HH\Proj^2$ be the point satisfying $\Pi = x_{3m+1}
x_{3m+1}^\dagger$. We know that $\{x_1,\dots,x_{3m}\}$ is a regular simplex,
as in Proposition~\ref{prp:12}. For $i=1,\dots,3m$ define $\Pi_i = x_i
x_i^\dagger$ and let $\alpha$ be the common inner product $\langle \Pi_i ,
\Pi_j \rangle$ (for $i,j \le 3m$ with $i \ne j$). By the definition of $\Pi$,
\begin{equation} \label{eq:prop13}
\Pi + \sum_{i=1}^{3m} \Pi_i = \frac{3m+1}{3} I_3.
\end{equation}
Since $\langle \Pi_i , \Pi_j \rangle = \alpha$ for $i \ne j$ and $\langle
\Pi_i,\Pi_i\rangle = 1$, the symmetry of \eqref{eq:prop13} implies that the
inner products $\langle \Pi, \Pi_i\rangle$ are all equal; call their common
value $\beta$.  Taking the inner product of \eqref{eq:prop13} with $\Pi$ and
$\Pi_i$ yields
\begin{align*}
1 + 3m \beta & = (3m+1)/3\quad\textup{and}\\
\beta + (3m-1)\alpha + 1 &= (3m+1)/3,
\end{align*}
respectively.  Subtracting shows that $\alpha=\beta$, so
$\{x_1,\dots,x_{3m+1}\}$ is a simplex, and it is tight by \eqref{eq:prop13}.
\end{proof}

We get a surjective Jacobian when we apply the conditions of the above
proposition in Theorem \ref{thm:implicit} with $m=4$, proving the
following result.

\begin{theorem} \label{thm:13}
There is a tight simplex of $13$ points in $\HH\Proj^2$.  In fact, there is
such a tight simplex with cyclic symmetry.
\end{theorem}

Theorems \ref{thm:12} and \ref{thm:13} establish the
existence of tight simplices, and their proof could
also provide the dimension of the space of
tight simplices with cyclic symmetry.  They cannot,
though, tell us the dimension of the full space of tight
simplices.

If Proposition \ref{prp:general} had applied then
we would have concluded that, in some neighborhood,
the space of tight simplices of $12$ (resp., $13$)
points in $\HH\Proj^2$ has dimension $7$ (resp., $4$).
The observed rank deficiency of two has several
possible explanations, including the following:
it might mean that two of the constraints are redundant, so
that the space of tight simplices is two dimensions larger than
predicted; it might mean that the constraints become degenerate at the
solutions, but the space of tight simplices is still a manifold;
or it might mean that the space of tight
simplices is not even locally a manifold.  Based
on numerical evidence (see \textsection\ref{subsec:dimensions}),
we conjecture that the first possibility holds.

\begin{conjecture} \label{conj:dim12and13}
There exists a $12$-point (resp., $13$-point)
tight simplex in $\HH\Proj^2$ such that, in a
neighborhood thereof, the space of
tight simplices has dimension $9$
(resp., $6$).
\end{conjecture}

\subsection{$15$-point simplices}
\label{subsec:15}

The case of $15$ points in $\HH \Proj^2$ is special for a few
reasons. First, it may be the only case in quaternionic
projective spaces where the cardinality upper bound in
Proposition~\ref{prp:bounds} is achieved (beyond $\HH\Proj^1$,
which is $S^4$ and clearly contains a $6$-point simplex). Also,
in comparison with the other cases in Tables~\ref{tbl:general3},
this case has especially large rank
deficiency.  This suggests that the moduli space of simplices is of a
larger dimension than $\sdim{15}{\HH\Proj^2}$. That turns out
to be correct, as we now show.

\begin{proposition} \label{prp:15}
Suppose $x_1,\dots,x_{15} \in \HH^3$ satisfy
\[\langle \Gamma_i , \Gamma_j \rangle = -\frac{1}{21} \quad \text{for }i \ne j,\]
where
\[\Gamma_i := x_i x_i^\dagger - \frac{1}{3} |x_i|^2 I_3.\]
Suppose additionally that $|x_i|^4 \in [1-10^{-6},1+10^{-6}]$ for each $i$.
Then $|x_i| = 1$ and $\{x_1,\dots,x_{15}\}$ is a tight simplex in
$\HH\Proj^2$.
\end{proposition}

We do not think the assumption $|x_i|^4 \in [1-10^{-6},1+10^{-6}]$ is
necessary for the proposition to hold, but it is easy to verify in our
applications and lets us prove the result with local calculations.  This
proof and that of Proposition~\ref{prp:27} will be based on two technical
lemmas (Lemmas~\ref{lemma:S} and~\ref{lemma:D}), which we defer until the end
of the section.  It would be straightforward to replace them in our
applications with bounds computed using interval arithmetic (see
\textsection\ref{subsec:finding}), but they are simple enough to prove by
hand, so we do so below.

\begin{proof}
For each $i$ write $|x_i|^4 = 1 + \delta_i$, and let $\delta = \max_i
|\delta_i|$.  It suffices to show $\delta=0$, because $\{x_1,\dots,x_{15}\}$
is then a tight simplex.  Specifically, define $\eta_i = (1+\delta_i)^{-1/2}$
and let $\Pi_i = x_i x_i^\dagger/|x_i|^2 = \eta_i x_i x_i^\dagger$ denote the
projection matrix associated to $x_i$. Then
\[
\langle \Pi_i , \Pi_j \rangle = \eta_i \eta_j \langle \Gamma_i , \Gamma_j \rangle + \frac{1}{3} =
\begin{cases} 1 & {\textup{if $i=j$, and}} \\ - \eta_i \eta_j/21 + 1/3 & {\textup{if $i\ne j$.}}
\end{cases}
\]
If $\eta_i=1$ for all $i$, then these inner products agree with the desired
value $2/7$ in a tight simplex of $15$ points.

Our strategy is to show that nonnegativity of the second zonal harmonic sum
forces $\delta=0$, given a rank condition coming from the fact that $15$
equals the dimension of the space of Hermitian matrices.

Recall that the zonal harmonics on $\HH\Proj^{d-1}$ are given by Jacobi
polynomials $P^{(2d-3,1)}_k(2t-1)$. Specifically, the functions
\[K_k(x,y) = P^{(2d-3,1)}_k(2|\langle x,y \rangle|^2-1)\]
are positive-definite kernels on $\HH\Proj^{d-1}$. Let $\Sigma_k$ be the sum
of the kernel $K_k(x,y)$ over the projective code determined by $\{x_1,
\dots, x_{15}\}$.  Then positive definiteness implies $\Sigma_k \ge 0$.

We will require only $\Sigma_2$.  As $P_2{(3,1)}(2t-1) = 28t^2 - 21t + 3$, we
can write $\Sigma_2$ in terms of the moments $\sum_{i,j=1}^{15} \langle \Pi_i
, \Pi_j \rangle^k$ with $k \le 2$.  If $\delta=0$, then $\Sigma_2=0$, and we
wish to compute it to second order in $\delta_1,\dots,\delta_{15}$ in terms
of the moments $m_1 := \sum_i \delta_i$ and $m_2 := \sum_i \delta_i^2$.
Applying Lemma~\ref{lemma:S} with $P_{i,j} = \langle \Pi_i , \Pi_j \rangle$,
we find that
\begin{equation} \label{eq:sigma2approx}
\left| \Sigma_2 - \left ( -\frac{10}{3} m_1 + \frac{23}{252} m_1^2 + \frac{719}{252} m_2 \right ) \right |
\le 8295 \cdot \delta^3.
\end{equation}

If we could approximate $\Sigma_2$ sufficiently well by a negative-definite
quadratic form in $\delta_1,\dots,\delta_{15}$, then $\Sigma_2 \ge 0$ would
imply $\delta=0$. However, the approximation in \eqref{eq:sigma2approx} is
not negative definite.  To make it so, we must add correction terms based on
additional constraints satisfied by the perturbations $\delta_i$.

These additional constraints come from a singular Gram matrix. We have
$\langle \Gamma_i , \Gamma_i \rangle = \tfrac{2}{3} (1+\delta_i)$, and the
Gram matrix of the elements $\sqrt{\tfrac{2}{3}} \Gamma_i$ is
\[
G = \begin{pmatrix}
1+\delta_1 & {} &  -\tfrac{1}{14} \\
{} & \ddots & {} \\
{-\tfrac{1}{14}} & {} & 1+\delta_N
\end{pmatrix}.
\]
Each of $\Gamma_1,\dots,\Gamma_{15}$ is a traceless Hermitian matrix, so they
must be linearly dependent, because the space of such matrices has dimension
$14$.  Thus, the Gram matrix $G$ must be singular.  Let $D := 14^{14}
\det(G)/15^{12}$ be its determinant, normalized as in Lemma~\ref{lemma:D}.
Of course $D=0$, but we know from Lemma~\ref{lemma:D} that
\[
| D - 15 m_1 - 14 (m_1^2 - m_2) | \leq 50625 \cdot \delta^3
\]
and
\[
| D^2 - 225 m_1^2 | \leq 4556250 \cdot \delta^3.
\]
Because $D$ (and so $D^2$) must vanish and $\Sigma_2$ must be nonnegative,
\[
\Sigma_2' := 4200D - 269 D^2 + 18900 \Sigma_2
\]
must be nonnegative as well.  However, from the above inequalities, we have
\[
|\Sigma_2'  + 4875m_2| \leq  16 \cdot 10^8 \cdot \delta^3.
\]
We have $-4875 m_2 \leq -4875 \cdot \delta^2$, and the assumption $\delta \le
10^{-6}$ implies that
\[
16 \cdot 10^8 \cdot \delta^3 \le 4875 \cdot \delta^2.
\]
It follows that $\Sigma_2' \le 0$, with equality iff $\delta=0$. Because
$\Sigma_2'$ is nonnegative, we conclude that indeed equality must hold, as
desired.
\end{proof}

Using this system of constraints, we do get a nonsingular Jacobian matrix and
hence we can apply Theorem \ref{thm:implicit}.  This yields a
$75$-dimensional solution space; after subtracting overcounting and
symmetries, we arrive at the following.

\begin{theorem} \label{thm:15}
There is a tight simplex of $15$ points in $\HH\Proj^2$.  In fact, locally
there is a $9$-dimensional space of such simplices.\footnote{As opposed to
the absurd $-5$ predicted by $\sdim{15}{\HH\Proj^2}$.}
\end{theorem}

Theorem~\ref{thm:15} establishes the existence of a tight $2$-design in
$\HH\Proj^2$.  The common inner product in this simplex is $2/7$, contrary to
a theorem of Bannai and Hoggar asserting that the inner products in tight
designs are always reciprocals of integers \cite[Corollary 1.7(b)]{BH}. The
case of $2$-designs is not addressed in their proof, and Bannai has informed
us that this was an oversight in the theorem statement.  See also \cite{L09} for another correction
(the icosahedron is a tight $5$-design in $\C\Proj^1$ with irrational inner
products).

It would be interesting to determine whether using the points of a $15$-point
simplex as vertices could lead to a minimal triangulation of $\HH\Proj^2$
(see \cite{BK}), as well as whether the same is true for a $27$-point simplex
in $\OO \Proj^2$.

Tight $2$-designs in $\HH\Proj^{d-1}$ are quaternionic analogues of SIC-POVMs
\cite{SICPOVM}.  Because SIC-POVMs seem to exist in $\C\Proj^{d-1}$ for every
$d$, it is natural to speculate that tight quaternionic $2$-designs should be
even more abundant, but we have not found any examples with $d>3$.

So far, we have shown that there are tight simplices in $\HH\Proj^2$ of every
size up to $15$ except for $14$.

\begin{conjecture} \label{conj:no14}
There does not exist a tight simplex of $14$ points in $\HH\Proj^2$.
\end{conjecture}

Similarly, we will see in \textsection\ref{oct} that there are tight
simplices in $\OO\Proj^2$ of every size up to $27$ except for $26$.
In $\R\Proj^2$ every size up to $6$ except for $5$ occurs, while in
$\C\Proj^2$ we see every size up to $9$ except for $5$ and $8$.  It
seems unlikely to be a coincidence that the second largest possible
size is always missing in projective planes, but we do not have a
proof beyond $\R\Proj^2$.  (As explained after
Lemma~\ref{lma:tightmeanstight}, linear programming bounds suffice to
disprove the existence of tight $5$-point simplices in $\R\Proj^2$.  However, they do not
rule out the analogous cases in $\C\Proj^2$, $\HH\Proj^2$, or
$\OO\Proj^2$.)

In the remainder of this section, we state and prove the deferred lemmas from
the proof of Proposition~\ref{prp:15}.

\begin{lemma} \label{lemma:S}
Given $d \ge 2$, $N>1$, and $\delta_1,\dots,\delta_N$ with $\delta := \max_i
|\delta_i| \le 1/4$, set $\eta_i = (1+\delta_i)^{-1/2}$, $\lambda =
-\tfrac{d-1}{d(N-1)}$, $m_1 = \sum_i \delta_i$, $m_2 = \sum_i \delta_i^2$,
and
\[
P_{i,j} =
\begin{cases} 1 & {\textup{if $i=j$, and}} \\ \eta_i \eta_j \lambda + 1/d & {\textup{if $i\ne j$.}} \end{cases}
\]
Then the moments $S_k := \sum_{i,j=1}^N P_{i,j}^k$ satisfy the following
bounds. Let
\begin{align*}
T_0 &= N^2, \\
T_1 &= \frac{N^2}{d} + \frac{d-1}{d} m_1 + \lambda\left(\frac{3N}{4} - 1\right)m_2 + \frac{\lambda}{4} m_1^2,\quad\text{and}\\
T_2 &= \frac{N^2(N+d^2-2d)}{d^2(N-1)} - \frac{2(N-d)\lambda }{d} m_1\\
& \quad {} + \lambda \left ( \lambda + \frac{1}{2d} \right) m_1^2 - \frac{\lambda}{d} \left( (2+ \lambda) d - \frac{3N}{2} \right) m_2.
\end{align*}
Then
\begin{align*}
S_0 &= T_0, \\
|S_1 - T_1| & \leq 5N \delta^3,\quad\text{and}\\
|S_2 - T_2| & \leq 16N \delta^3.
\end{align*}
\end{lemma}

\begin{proof}
It is clear that $S_0 = N^2$.  For $S_1$ and $S_2$, we begin by explicitly
computing that $S_1$ equals
\[
\lambda \left ( \sum_{i=1}^N \eta_i \right )^2 - \lambda \left ( \sum_{i=1}^N \eta_i^2 \right ) + \frac{N^2-N+Nd}{d}
\]
and $S_2$ equals
\[
\lambda^2 \left ( \sum_{i=1}^N \eta_i^2 \right )^2 -
\lambda^2 \left ( \sum_{i=1}^N \eta_i^4 \right ) + \frac{2\lambda}{d} \left (
\sum_{i=1}^N \eta_i \right )^2 - \frac{2\lambda}{d} \left ( \sum_{i=1}^N
\eta_i^2 \right ) + \frac{N^2 + Nd^2 - N}{d^2}.
\]

Now, using $\eta_i = (1+\delta_i)^{-1/2}$and $\delta \le 1/4$, Taylor's
theorem with the Lagrange form of the remainder yields the estimates
\[ \left| \eta_i^a - \left ( 1 - \frac{a}{2} \delta_i + \frac{a(a+2)}{4} \cdot \frac{\delta_i^2}{2} \right) \right|
\le \frac{4}{81} \left(\frac{4}{3}\right)^{a/2} a(a+2)(a+4) \cdot \delta^3 \]
for all $a > 0$.  Taking $a=1,2,4$, we get the bounds
\begin{align*}
\left| \sum_{i=1}^N \eta_i - \bigg( N - \frac{m_1}{2} + \frac{3}{8} m_2 \bigg) \right| & \le  N\delta^3, \\
\left| \sum_{i=1}^N \eta_i^2 - \bigg( N - m_1 + m_2 \bigg) \right| & \le 4N \delta^3, \\
\left| \sum_{i=1}^N \eta_i^4 - \bigg( N - 2m_1 + 3m_2 \bigg) \right| & \le 17 N \delta^3,
\end{align*}
respectively.

We also have the simple bounds $|m_i| \le N\delta^i$. Using these, we find
\begin{align*} % ad hoc spacing
&\left| \left( \sum_{i=1}^N \eta_i \right)^2 - \left ( N^2 - N m_1 + \frac{3}{4} N m_2 + \frac{m_1^2}{4} \right) \right| \\
& \quad \le \left| \left( \sum_{i=1}^N \eta_i \right)^2 - \left ( N - \frac{m_1}{2} + \frac{3}{8} m_2 \right)^2 \right| +
\left| \frac{3}{8} m_1 m_2 - \frac{9}{64} m_2^2 \right| \\
&\quad \le N \delta^3 \cdot \left|\sum_{i=1}^N \eta_i + N - \frac{m_1}{2} + \frac{3}{8} m_2 \right|  + N^2 \left ( \frac{3}{8} \cdot \delta^3 + \frac{9}{64} \cdot \delta^4  \right) \\
&\quad \le N \delta^3 \cdot \left( 2\left|N - \frac{m_1}{2} + \frac{3}{8} m_2\right| + N \delta^3 \right) + N^2 \left ( \frac{3}{8} \cdot \delta^3 + \frac{9}{64} \cdot \delta^4  \right) \\
&\quad \le N \delta^3 \cdot \left( 2N \left(1 + \frac{1}{2}\delta + \frac{3}{8} \delta^2 \right) + N \delta^3 \right) + N^2 \left ( \frac{3}{8} \cdot \delta^3 + \frac{9}{64} \cdot \delta^4  \right) \\
&\quad  \le 3N^2 \delta^3.
\end{align*}
We similarly compute
\[ \left| \left( \sum_{i=1}^N \eta_i^2 \right)^2 - \left ( N^2 - 2 N m_1 + 2 N m_2 + m_1^2 \right) \right| \le 13 N^2 \delta^3 . \]
Combining all of these estimates with $d \ge 2$, $N \ge 2$, and $|\lambda|
\le 1/N$ leads to bounds of $(3N+4)\delta^3$ and $(3N+17+17/N)\delta^3$ for
$|S_1-T_1|$ and $|S_2-T_2|$, respectively.  We have rounded them up to
pleasant multiples of $N$ in the lemma statement.
\end{proof}

\begin{lemma} \label{lemma:D}
Suppose $N>3$, and let
\[
D = \frac{(N-1)^{N-1}}{N^{N-3}} \det \begin{pmatrix}
1+\delta_1 & {} & -\tfrac{1}{N-1}\\
{} & \ddots & {} \\
-\tfrac{1}{N-1} & {} & 1+\delta_N
\end{pmatrix},
\]
where every off-diagonal entry in the above matrix equals $-1/(N-1)$. Set
$\delta = \max_i |\delta_i|$, $m_1 = \sum_i \delta_i$, and $m_2 = \sum_i
\delta_i^2$. If $\delta \le 1/(2N)$, then \[ |D - Nm_1 - (N-1)(m_1^2 - m_2)|
\le N^4 \delta^3
\]
and
\[
\left | D^2 - N^2 m_1^2 \right | \le
6N^5 \delta^3.
\]
\end{lemma}

\begin{proof}
Let $G_r$ be the $r \times r$ matrix with diagonal entries $1$ and
off-diagonal entries $\beta$. It is easy to show\footnote{See the footnote in the proof of Proposition~\ref{prp:bounds}.} that
\[
D_r := \det(G_r) = \big(1 + (r-1)\beta \big) (1-\beta)^{r-1}.
\]
Setting $\beta = -1/(N-1)$, we have
\[
D_r = \frac{(N-r)N^{r-1}}{(N-1)^r}.
\]
Using this, for
\[
G = \begin{pmatrix}
1+\delta_1 & {} & -\tfrac{1}{N-1} \\
{} & \ddots & {} \\
-\tfrac{1}{N-1} & {} & 1+\delta_N
\end{pmatrix}
\]
we find that
\begin{align*}
\det(G) &= D_N + \Big(\sum_i \delta_i\Big) D_{N-1} +
\Big(\sum_{i < j} \delta_i \delta_j\Big) D_{N-2} + \dots + \prod_i \delta_i\\
&= 0 + \Big(\sum_i \delta_i\Big) \frac{N^{N-2}}{(N-1)^{N-1}} +
\Big(\sum_{i < j} \delta_i \delta_j\Big) \frac{2 N^{N-3}}{(N-1)^{N-2}} + \dots + \prod_i \delta_i.
\end{align*}
In terms of the moments $m_1 = \sum_i \delta_i$ and $m_2 = \sum_i
\delta_i^2$, the rescaled determinant $D = (N-1)^{N-1}\det(G)/N^{N-3}$
satisfies
\[
|D - Nm_1 - (N-1)(m_1^2 - m_2)| \leq \sum_{k \geq 3} \binom{N}{k}\delta^k  N^{2-k} (N-1)^{k-1} k.
\]
The $k=3$ term on the right is $(N-1)^3(N-2)\delta^3/2 \le N^4\delta^3/2$.
Because $\delta \le 1/(2N)$, each subsequent term diminishes by a factor of
at least $1/2$.  Thus, summing the geometric series, we have
\[
|D - Nm_1 - (N-1)(m_1^2 - m_2)| \le N^4 \delta^3.
\]
Note that the trivial bounds $|m_1| \le N \delta$ and $m_2 \le N \delta^2$
imply
\begin{align*}
|Nm_1 + (N-1)(m_1^2 - m_2)| & \leq N^2\delta + (N-1)( N^2 \delta^2 + N \delta^2) \\
&= N^2 \delta + N(N^2 - 1)\delta^2 \\
& \leq N^2 \delta + N^3 \delta^2 \leq 2N^2 \delta
\end{align*}
and therefore $|D| \leq 2 N^2 \delta + N^4 \delta^3 \leq 3N^2 \delta$. Now,
to control $D^2$, we write
\begin{align*}
\big|D^2 - (Nm_1 + (N-1)(m_1^2-m_2))^2\big| &\le N^4 \delta^3 |D + Nm_1 + (N-1)(m_1^2 - m_2)|\\
& \le N^4\delta^3 (|D| + |Nm_1 + (N-1)(m_1^2 - m_2)|) \\
& \le N^4 \delta^3 (5N^2 \delta) = 5 N^6 \delta^4.
\end{align*}
Furthermore,
\begin{align*} % ad hoc spacing
& |N^2m_1^2 - (Nm_1 + (N-1)(m_1^2-m_2))^2|  \\
& \qquad \le 2N(N-1)|m_1|(m_1^2+m_2) + (N-1)^2(m_1^2+m_2)^2 \\
& \qquad \le 2N(N-1) N \delta (N^2 \delta^2 + N\delta^2) + (N-1)^2 (N^2 \delta^2 + N\delta^2)^2 \\
& \qquad = 2N^3(N^2-1) \delta^3 + N^2(N^2-1)^2 \delta^4 \leq 3N^5 \delta^3.
\end{align*}
Combining these two bounds with the triangle inequality and using $N\delta
\leq 1/2$, we obtain the asserted bound for $|D^2 - N^2 m_1^2|$.
\end{proof}

\section{Simplices in $\OO\Proj^2$} \label{oct}

The study of simplices in $\OO\Proj^2$ unfolds much like that in
$\HH\Proj^2$; we get essentially the same results as long as we take care to
work in an affine chart. In particular, we can handle the generic case, $24$-
and $25$-point simplices, and $27$-point simplices using adaptations of
Propositions \ref{prp:general}, \ref{prp:12} and \ref{prp:13}, and
\ref{prp:15}, respectively.

\subsection{Generic case}

\begin{proposition} \label{prp:generalpanda}
For $i=1,\dots,N$, suppose $x_i = (a_i,b_i,c_i) \in \R_{+} \times \OO^2$ and
$w_i \in \R$ satisfy
\begin{enumerate}
\item $|a_i|^2 + |b_i|^2 + |c_i|^2 = 1$ for $i=1,\dots,N$,
\item $\rho(x_i,x_j)^2 = \rho(x_{i'},x_{j'})^2$ for $1 \le
    i < j \le N$ and $1 \le i' < j' \le N$, and
\item $\displaystyle \sum_{i=1}^N w_i \begin{pmatrix} a_i \\ b_i \\ c_i \end{pmatrix} \begin{pmatrix} \bar a_i & \bar b_i & \bar c_i \end{pmatrix} = \begin{pmatrix} 1 & 0 & 0 \\ 0 & 1 & 0 \\ 0 & 0 & 1 \end{pmatrix}$.
\end{enumerate}
Then $w_1 = \cdots = w_N = 3/N$ and $\{x_1,\dots,x_N\}$ is a tight simplex.
\end{proposition}

We omit the proof of Proposition \ref{prp:generalpanda}
as it is nearly identical to that of Proposition \ref{prp:general}.

We can attempt to apply Proposition \ref{prp:generalpanda} with
Theorem \ref{thm:implicit} just as we did for simplices in
quaternionic projective spaces.  There are
\[N + \left(\frac{N(N-1)}{2}-1\right) + 27\text{ real constraints}\quad\text{in}\quad 18N\text{ real variables},\]
so, when the Jacobian is nonsingular, we get a solution space of dimension
$(N-1)(34-N)/2 - 9$.  As before, we should deduct the dimension of the
symmetry group.  The symmetry group of $\OO\Proj^2$ is the exceptional Lie
group $F_4$, which has dimension $52$.  Thus, our final expression for the
expected local dimension of the moduli space of simplices is
\[\sdim{N}{\OO\Proj^2} := \frac{(N-1)(34-N)}{2} - 61.\]
Again, as with $\sdim{N}{\HH\Proj^{d-1}}$, this formula only applies when, at
our numerical solution, Theorem \ref{thm:implicit} applies to the conditions
of Proposition \ref{prp:generalpanda} and the simplex has zero-dimensional
stabilizer.

\begin{theorem} \label{thm:generalpanda}
For the values of $N$ listed in part (a) of Table
\ref{tbl:generalop2}, there exist tight $N$-point simplices in
$\OO\Proj^2$.
\end{theorem}

\begin{table}
\caption{Cases in $\OO\Proj^2$: (a) proven existence of tight simplices; (b)
singular Jacobian; (c) conjectured nonexistence.} \label{tbl:generalop2}

\centering
\begin{tabular}{c}
\begin{tabular}[t]{cccccccc}
\toprule
$N$ & $\sdim{N}{\OO\Proj^2}$ & {\hspace{8pt}} & $N$ & $\sdim{N}{\OO\Proj^2}$ & {\hspace{8pt}} & $N$ & $\sdim{N}{\OO\Proj^2}$ \\
\midrule
 5 & 0${}^{\dagger}$      & {}     &  12 & 60      & {}     &  19 & 74   \\
 6 &  9      & {}     &  13 & 65      & {}     &  20 & 72   \\
 7 & 20      & {}     &  14 & 69      & {}     &  21 & 69   \\
 8 & 30      & {}     &  15 & 72      & {}     &  22 & 65   \\
 9 & 39      & {}     &  16 & 74      & {}     &  23 & 60   \\
10 & 47      & {}     &  17 & 75      & {}                  \\
11 & 54      & {}     &  18 & 75      & {} \\
\bottomrule
\end{tabular} \vspace{4pt} \\
{(a)}
\end{tabular}

\bigskip

\begin{tabular}{>{\centering\arraybackslash\hspace{0pt}}p{0.33\textwidth}>{\centering\arraybackslash\hspace{0pt}}p{0.2\textwidth}}
{\begin{tabular}[t]{cc}
\toprule
$N$ & rank deficiency \\
\midrule
24 & 2 \\
25 & 2 \\
27 & 26 \\
\bottomrule
\end{tabular}} &
{\begin{tabular}[t]{c}
\toprule
$N$ \\
\midrule
26 \\
\bottomrule
\end{tabular}} \\
{\vspace{-8pt}

(b)}
&
{\vspace{-8pt}

(c)} \\
\end{tabular}

\bigskip

\begin{minipage}{0.7\textwidth}
\footnotesize $\dagger$ Actually $\sdim{5}{\OO\Proj^2}$ is not $0$; rather,
it equals $-3$. This is the only case in which the simplex we found has a
positive-dimensional stabilizer. The stabilizer is $3$-dimensional, so the
actual dimension of the moduli space, which is what
$\sdim{5}{\OO\Proj^2}$ is really intended to capture, is $0$.
\end{minipage}
\end{table}

\subsection{$24$- and $25$-point simplices}

The following proposition is proven similarly to Proposition \ref{prp:12}.

\begin{proposition} \label{prp:24} Let $\sigma$ be the cyclic-shift
  automorphism $\sigma(a,b,c) = (b,c,a)$. Suppose $x_1,\dots,x_{3m}
  \in \OO^3$ and $w_1,\dots,w_{3m} \in \R$ satisfy the
  following conditions:
\begin{enumerate}
\item $x_{m+i} = \sigma(x_i)$ for $i=1,\dots,2m$,
\item $w_{m+i} = w_i$ for $i=1,\dots,2m$,
\item $x_i \in \R_{+} \times \OO^2$ and $|x_i|^2 = 1$ for $i=1,\dots,m$,
\item the squared distances $\rho(x_i,x_j)^2$ for
    $i=1,\dots,m$ and the following values of $j$ are all
    equal: \caseup{(i)} $j=i+m$, \caseup{(ii)} $i < j \le
    m$, \caseup{(iii)} $i+m < j \le 2m$, \caseup{(iv)}
    $i+2m < j \le 3m$, and
\item the matrix $\sum_{i=1}^{3m} w_i x_i x_i^\dagger$ has $1,1$
entry equal to $1$ and vanishing $1,2$ entry.
\end{enumerate}
Then $w_1 = \cdots = w_{3m} = 1/m$ and $\{x_1,\dots,x_{3m}\}$ is a tight
simplex.
\end{proposition}

Using the conditions of Proposition \ref{prp:24} with $m=8$ in Theorem
\ref{thm:implicit} yields a surjective Jacobian, allowing us to prove
the following theorem.

\begin{theorem}
There is a tight simplex of $24$ points in $\OO\Proj^2$.  In fact, there is
such a tight simplex with cyclic symmetry.
\end{theorem}

Similarly, to prove the existence of tight simplices with $25$ points,
we use the following adaptation of Proposition \ref{prp:13}.

\begin{proposition} \label{prp:25} Let $\sigma$ be the cyclic-shift
  automorphism $\sigma(a,b,c) = (b,c,a)$.  Suppose $x_1,\dots,x_{3m}
  \in \OO^3$ satisfy the following conditions:
  \begin{enumerate}
  \item $x_{m+i} = \sigma(x_i)$ for $i=1,\dots,2m$,
  \item $x_i \in \R_{+} \times \OO^2$ and $|x_i|^2 = 1$ for
      $i=1,\dots,m$,
  \item the squared distances $\rho(x_i,x_j)^2$ for
      $i=1,\dots,m$ and the following values of $j$ are all
      equal: \caseup{(i)} $j=i+m$, \caseup{(ii)} $i < j \le
      m$, \caseup{(iii)} $i+m < j \le 2m$, \caseup{(iv)}
      $i+2m < j \le 3m$, and
  \item the $1,2$ entry of the matrix $\sum_{i=1}^{3m} x_i
    x_i^\dagger$ has real part $1/6$ and magnitude $1/3$.
  \end{enumerate}
  Then there is a unique point $x_{3m+1} \in \OO\Proj^2$ such that
  $\{x_1,\dots,x_{3m},x_{3m+1}\}$ is a tight simplex, and that point
  satisfies $\sigma(x_{3m+1}) = x_{3m+1}$.
\end{proposition}

Using the conditions above with $m = 8$ in Theorem \ref{thm:implicit}
yields a surjective Jacobian.

\begin{theorem}
  There is a tight simplex of $25$ points in $\OO\Proj^2$.  In
  fact, there is such a tight simplex with cyclic symmetry.
\end{theorem}

Continuing the correspondence with $12$- and $13$-point
simplices in $\HH\Proj^2$, based on numerical evidence
we conjecture the following.

\begin{conjecture} \label{conj:dim24and25}
There exists a $24$-point (resp., $25$-point)
tight simplex in $\OO\Proj^2$ such that, in a
neighborhood thereof, the space of
tight simplices has dimension $56$
(resp., $49$).
\end{conjecture}

\subsection{$27$-point simplices}

\begin{proposition} \label{prp:27}
Suppose $x_i = (a_i,b_i,c_i) \in \R_{+} \times \OO^2$ satisfy
\[\langle \Gamma_i , \Gamma_j \rangle = -\frac{1}{39} \quad \text{for }i \ne j,\]
where
\[\Gamma_i := \begin{pmatrix} a_i \\ b_i \\ c_i \end{pmatrix} \begin{pmatrix} \bar a_i & \bar b_i & \bar c_i \end{pmatrix} - \frac{1}{3} (a_i^2+|b_i|^2+|c_i|^2) I_3.\]
Suppose additionally that $|x_i|^4 \in [1-10^{-7},1+10^{-7}]$ for each $i$.
Then $|x_i| = 1$ and $\{x_1,\dots,x_{27}\}$ determines a tight simplex in
$\OO\Proj^2$.
\end{proposition}

\begin{proof}
We use the same proof technique as Proposition \ref{prp:15}, with the only
difference being the constants appearing in the proof.

As before, we write $|x_i|^4 = 1 + \delta_i$, and let $\delta = \max_i
|\delta_i|$.  Let $G$ be the Gram matrix of $\sqrt{\tfrac{2}{3}} \Gamma_1,
\dots \sqrt{\tfrac{2}{3}} \Gamma_{27}$.  Then $\det(G)=0$, and by
Lemma~\ref{lemma:D} the normalized determinant $D := 26^{26} \det(G)/27^{24}$
satisfies
\[
|D - 27m_1 - 26(m_1^2-m_2)| \le 531441 \cdot \delta^3
\]
and
\[
|D^2 - 729 m_1^2| \le 86093442 \cdot \delta^3.
\]

The second zonal harmonic on $\OO\Proj^2$ is given by the Jacobi polynomial
\[P_2^{(7,3)}(2t-1) = 91t^2 - 65t + 10.\]
Let $\Sigma_2$ be the sum of the kernel $K_2(x,y) := P^{(7,3)}_2(2|\langle
x,y \rangle|^2-1)$ over the projective code determined by $\{x_1, \dots,
x_{27}\}$, so $\Sigma_2 \ge 0$. By Lemma~\ref{lemma:S},
\[
\left|\Sigma_2-\left( -6 m_1 + \frac{41}{468} m_1^2 + \frac{2429}{468} m_2 \right) \right|
\leq  48087 \cdot \delta^3.
\]
Because $D=0$,
\[\Sigma_2' := 75816 D - 2745 D^2 + 341172 \Sigma_2\]
must be nonnegative, but $|\Sigma_2' + 200475 m_2| \leq 293024167110 \cdot
\delta^3$. Because $m_2 \ge \delta^2$, when $\delta \le 10^{-7}$ we have
$\Sigma_2' \le 0$ with equality only when $\delta=0$.  Thus, $\delta=0$ and
$\{x_1,\dots,x_{27}\}$ determines a tight simplex in $\OO\Proj^2$.
\end{proof}

Applying Theorem \ref{thm:implicit} with the conditions of the above
proposition, we find a suitable point for which the Jacobian is
surjective.

\begin{theorem} \label{thm:27}
There is a tight simplex of $27$ points in $\OO\Proj^2$.  In fact, locally there is a
$56$-dimensional space of such simplices.
\end{theorem}

Theorem \ref{thm:27} establishes the existence of a tight $2$-design in
$\OO\Proj^2$.  Such designs were previously conjectured not to exist
\cite[p.~251]{Ho1}.  It is known \cite{Ho3} that tight $t$-designs in
$\OO\Proj^2$ can only exist for $t=2$ and $t=5$, and there is an explicit
construction of a $819$-point tight $5$-design \cite{Coh}, so
Theorem~\ref{thm:27} completes the list of $t$ for which tight $t$-designs
exist in $\OO\Proj^2$.

\begin{conjecture} \label{conj:no26}
There does not exist a tight simplex of $26$ points in
$\OO\Proj^2$.
\end{conjecture}

See also the discussion after Conjecture~\ref{conj:no14}.

\section{Simplices in real Grassmannians} \label{grass}

Our techniques also apply to show the existence of many
simplices in Grassmannian spaces. The real Grassmannian $G(m,n)$
is the space of all $m$-dimensional subspaces of
$\R^n$. It is a homogeneous space for the orthogonal group
$O(n)$, isomorphic to $O(n)/(O(m) \times O(n-m))$, and it has
dimension $m(n-m)$. These spaces generalize (real) projective
space $\R\Proj^{n-1}$, which is the space of lines in $\R^n$.
The spaces $G(m,n)$ and $G(n-m,n)$ can be identified by
associating to each subspace its orthogonal complement, so in
what follows we always assume $m \le n/2$.

Though Grassmannians are generally not $2$-point homogeneous
spaces, there are still linear programming bounds
\cite{B,BBL08}. Here we will just consider the special case of the
simplex bound.

When $m \le n/2$, a pair of points in $G(m,n)$ is described by
$m$ parameters, namely the \emph{principal angles} between the
$m$-dimensional subspaces.  Given two $m$-dimensional subspaces
$U$ and $U'$, define sequences of unit vectors $u_1,\dots,u_m
\in U$ and $u'_1,\dots,u'_m \in U'$ inductively so that
$\langle u_i, u'_i \rangle$ is maximized subject to $\langle
u_i, u_j \rangle = \langle u'_i, u'_j \rangle = 0$ for $j < i$.
Then the principal angles are $\theta_i := \arccos \langle u_i,
u'_i \rangle$.

The \emph{chordal distance} on $G(m,n)$ is given by
\[
d_c(U,U') = \sqrt{\sin^2 \theta_1 + \dots + \sin^2 \theta_m }.
\]
Unlike in projective space, the chordal metric on Grassmannians is generally
not equivalent to the geodesic metric $\sqrt{\theta_1^2 + \dots +
\theta_m^2}$.  See \cite{CHS} for discussion of why the chordal metric is
preferable.

A \emph{generator matrix} for an element of $G(m,n)$ is a $m
\times n$ matrix whose rows form an orthonormal basis of the
subspace.  Given a generator matrix $X$, the orthogonal
projection onto the subspace is $X^t X$. Suppose $X_1$ and
$X_2$ are generator matrices for the subspaces $U_1$ and $U_2$,
and let $\Pi_i = X_i^t X_i$ (for $i=1,2$) be the orthogonal
projection matrices. Then the singular values of the matrix
$X_1 X_2^t$ are $\cos \theta_i$ for $1 \le i \le m$. It follows
that
\begin{equation} \label{eq:grassdist}
d_c(U_1, U_2)^2 = \frac{1}{2} || \Pi_1 - \Pi_2||_F^2 = m - \langle \Pi_1, \Pi_2 \rangle.
\end{equation}
Let $\Pi^0 = \Pi - (m/n) I_n$ be the traceless part of the
projection matrix. It can be thought of as a point in
$\R^{D}$, where $D = m(m+1)/2 - 1$, if we view $\R^{D}$ as the
space of trace-zero symmetric matrices.  It is easily checked
that $|| \Pi^0 ||_F^2 = m(n-m)/n$.
Therefore we obtain an isometric
embedding $U \mapsto \Pi^0$ of $G(m,n)$ into the sphere of
radius $\sqrt{m(n-m)/n}$ in $\R^{D}$ under the chordal metric.
The simplex bound for spherical codes gives us the following
result.

\begin{proposition}[Conway, Hardin, and Sloane \cite{CHS}] \label{prp:grassmannianbound}
Every $N$-point simplex in $G(m,n)$ satisfies
\[
N \le \binom{m+1}{2},
\]
and every code of $N$ points has squared chordal distance at
most
\[
\frac{m(n-m)}{n} \cdot \frac{N}{N-1}.
\]
\end{proposition}

This squared chordal distance is equivalent to having inner product
\begin{equation} \label{eq:grassalpha}
\frac{m(Nm-n)}{n(N-1)}
\end{equation}
between orthogonal projection matrices.

\begin{remark}
The $m=1$ case of Proposition~\ref{prp:grassmannianbound} is
the same as the $K=\R$ case of Proposition~\ref{prp:bounds}
(together with Proposition~\ref{prp:tightoptimal}). Indeed, the
proofs of these two results are essentially the same; they are
just phrased in different language.
\end{remark}

We say that a simplex in $G(m,n)$ is tight if its minimal chordal distance
meets the upper bound above.  Analogously to simplices in projective space, a
Grassmannian simplex is tight iff it is a $1$-design (i.e., a $2$-design in
the terminology of \cite{BCN}), which holds iff the
linear programming bound is sharp \cite{B}.  If the projection matrices of
the simplex are $\Pi_1,\dots,\Pi_N$, then another equivalent condition for
tightness is $\sum_{i=1}^N\Pi_i = (Nm/n) I_n$.

Conway, Hardin, and Sloane \cite{CHS} reported a number of putative tight
simplices based on numerical evidence, but except for a few explicit
constructions they did not present any techniques for rigorous existence
proofs.  (As in non-real projective spaces, it is not easy to reconstruct an
exact Grassmannian simplex from a numerical approximation.) The cases with
explicit constructions are listed in Table~\ref{tbl:grassexplicit}. By
applying our methods, we can certify the existence of simplices for many of
the cases previously identified but not settled.

\begin{proposition} \label{prp:grass}
Suppose $\{x_{i,j} \in \R^n\}_{\substack{i=1,\dots,N\\j=1,\dots,m}}$ and
$w_1,\dots,w_N$ satisfy the following conditions:
\begin{enumerate}
\item $|x_{i,j}| = 1$ for all $i,j$,
\item for all $i$ and all $j < j'$, $\langle x_{i,j} ,
    x_{i,j'} \rangle = 0$,
\item the inner products $\langle \sum_{j=1}^m x_{i,j}
    x_{i,j}^t , \sum_{j=1}^m x_{i',j} x_{i',j}^t \rangle$
    are equal for all distinct pairs $i,i'$, and
\item $\sum_{i=1}^N w_i \left ( \sum_{j=1}^m x_{i,j}
    x_{i,j}^t \right ) = I_n$.
\end{enumerate}
Then $w_1 = \cdots = w_N = n/(N m)$ and the subspaces $\Span
\{x_{i,1},\dots,x_{i,m}\}$ form a tight simplex in $G(m,n)$.
\end{proposition}

\begin{proof}
For each $i$, define $\Pi_i = \sum_{j=1}^m x_{i,j} x_{i,j}^t$. Because
$\{x_{i,1},\dots,x_{i,m}\}$ is an orthonormal system, this is the projection
matrix associated to the plane $\Span \{x_{i,1},\dots,x_{i,m}\}$.  Using
\eqref{eq:grassdist}, the third condition guarantees that we have a simplex.
Arguing as in the proof of Proposition \ref{prp:general}, we deduce from the
last condition that $w_1 = \cdots = w_N = n/(N m)$. Thus $\sum_{i=1}^N \Pi_i
= (N m/n) I_n$; as noted above, this is equivalent to tightness.
\end{proof}

In many cases the system specified by Proposition \ref{prp:grass}
is nonsingular, allowing us to apply Theorem \ref{thm:implicit}.
This yields the following.

\begin{theorem} \label{thm:grass}
For the values of $(N,m,n)$ listed in the ``proven'' column of
Table \ref{tbl:grass}, there exist tight $N$-point simplices in
$G(m,n,\R)$.
\end{theorem}

In the context of Proposition \ref{prp:grass}, we have
$N m n + N$ real variables and
\[N \cdot \binom{m+1}{2} + \left ( \frac{N(N-1)}{2} - 1 \right ) + \binom{n+1}{2}\]
real constraints.  Thus, when Theorem \ref{thm:implicit}
applies, we locally get a solution space whose dimension is the
difference of these counts.  Because $O(m)$ acts on the
different representations of each plane, we are overcounting
the dimension by $N \cdot \binom{m}{2}$.  Moreover, when the
symmetry group $O(n)$ of $G(m,n)$ acts with finite stabilizer
on the simplex, we should deduct $\binom{n}{2}$ from our final
dimension count.  Putting this all together, when these
conditions are satisfied (as in Remark \ref{rmk:whenrapplies}),
we get a neighborhood in which the moduli space of simplices has
dimension
\begin{equation} \label{eq:grassdim}
\sdim{N}{G(m,n)} := N m n - \frac{N(N-3)}{2} - N m^2 - n^2 + 1.
\end{equation}
As in projective spaces, we expect to find tight simplices in
most cases where $\sdim{N}{G(m,n)}>0$.  This parameter counting
argument heuristically explains the large number of tight
simplices found in \cite{CHS}.

We tested all cases up to dimension $n=8$, using our own software to
search for numerical solutions and also comparing with the numerical results
of Conway, Hardin, and Sloane \cite{CHS}.
As with simplices in projective spaces, sometimes the system specified
by Proposition \ref{prp:grass} was singular, and sometimes the numerical
evidence was unclear (as we saw in Tables~\ref{tbl:general3} and \ref{tbl:general5}, respectively).
These cases are in the third and fourth columns, respectively, of Table \ref{tbl:grass}.

\begin{table}
\caption{Previously known tight simplices with explicit
constructions in $G(m,n)$ for $n \le 8$.}
\label{tbl:grassexplicit}

\begin{center}
\begin{tabular}{ccc}
\toprule
$(m,n)$ & $N$ & {Reference} \\
\midrule
$(2,4)$ & 2--6 & \cite[pp.~145--146]{CHS} \\
$(2,4)$ & 10 & \cite[p.~147]{CHS} \\
$(2,6)$ & 9 & \cite[p.~154]{CHS} \\
$(3,7)$ & 28 & \cite[p.~152]{CHS} \\
$(2,8)$ & 8 & \cite[p.~154]{CHS} \\
$(2,8)$ & 20 & \cite[p.~135]{CHRSS} \\
$(2,8)$ & 28 & \cite[p.~154]{CHS} \\
\bottomrule
\end{tabular}
\end{center}
\vskip 0.5cm % To separate from the next table.  Feel free to remove.
\end{table}

\begin{table}
\caption{Tight Grassmannian simplices in $G(m,n)$.} \label{tbl:grass}
\begin{center}
\begin{tabular}{cccc}
\toprule
$(m,n)$ & Proven & Singular Jacobian & Ambiguous \\
\midrule
$(2,4)$ & 4--6 & 2, 3, 7, 8, 10 & \\
$(2,5)$ & 5--10 & 4, 11 & \\
$(2,6)$ & 5--14 & 3, 4 & \\
$(3,6)$ & 5--16 & 2--4 & 17 \\
$(2,7)$ & 6--17 & {} & 18 \\
$(3,7)$ & 5--22 & 4, 28 & 23 \\
$(2,8)$ & 6--21 & 4, 5, 28 & \\
$(3,8)$ & 5--28 & 4 & \\
$(4,8)$ & 5--30 & 2--4 & \\
\bottomrule
\end{tabular}
\end{center}
\end{table}

In addition to our existence proofs and the previously known
explicit constructions, several Grassmannian tight simplices
can be proven to exist using the following observation: if
there is a tight $N$-point simplex in $G(m,n)$ for some $m,n$,
then there is a tight $N$-point simplex in $G(km,kn)$ for all
$k \ge 1$.  This is immediate from block repetition
\cite[Proposition~12]{Cr}.  It proves existence for $11$ of the
singular cases in Table \ref{tbl:grass}. This leaves us with
only $7$ hitherto unresolved cases in which there is strong
numerical evidence for a tight simplex: $4$-point simplices in
$G(2,5)$, $G(3,6)$, $G(3,7)$, and $G(3,8)$; $7$- and $8$-point
simplices in $G(2,4)$; and $11$-point simplices in $G(2,5)$.
For completeness, we will settle all of these in the following
subsection.

We anticipate no difficulty in applying our techniques to
complex or quaternionic Grassmannians, but we have not done so.

\subsection{Miscellaneous special cases in Grassmannians} \label{subsec:miscgrass}

We begin with the case of $11$-point tight simplices in
$G(2,5)$. This can handled in the same way as $13$-point tight
simplices in $\HH\Proj^2$ and $25$-point tight simplices in
$\OO\Proj^2$; i.e., we can prove existence of simplices with
cyclic symmetry.  We will state the analogous result in greater
generality than we attempted in Proposition \ref{prp:13} (which
was written in the special case of $\HH\Proj^2$ rather than a
general projective space $\HH\Proj^{d-1}$), at the cost of some
additional complexity.

\begin{proposition} \label{prp:11inG25}
Fix dimensions $n > m > 0$ and let $\sigma$ be the cyclic-shift automorphism
$\sigma(x_1,x_2,\dots,x_n) = (x_2,\dots,x_n,x_1)$ on $\R^n$. Set $N = nk+1$
and suppose we have vectors $\{x_{i,j} \in
\R^n\}_{\substack{i=1,\dots,nk\\j=1,\dots,m}}$. For each $i$, define $\Pi_i =
\sum_{j=1}^m x_{i,j} x_{i,j}^t$. Define $\Pi_N = \tfrac{Nm}{n} I_n - \sum_{i
< N} \Pi_i$. Suppose that, for some $\eta \in
\big(\tfrac{m}{m+1},\tfrac{m}{m-1}\big)$, the following conditions are
satisfied:
\begin{enumerate}
\item $x_{k+i,j} = \sigma(x_{i,j})$ for all $i \le (n-1)k$ and all $j$,
\item $|x_{i,j}| = 1$ for all $i \le k$ and all $j$,
\item for all $i \le k$ and all $j < j'$, $ \langle
    x_{i,j}, x_{i,j'} \rangle = 0$,
\item the inner products $\langle \Pi_i , \Pi_{i'} \rangle$ are all equal
    for \caseup{(i)} $i \le k$, $i'=i+qk$, and $q=1,\dots,\lfloor
    \tfrac{n}{2}\rfloor$ and \caseup{(ii)} $i \le k-1$, $i'=i'_0+qk$, $i
    < i'_0 \le k$, and $q=0,\dots,n-1$, and
\item the first $\lfloor \tfrac{n}{2} \rfloor+1$ entries in the first row
    of $\Pi_N^2 - \eta \Pi_N$ are all zero.
\end{enumerate}
Then $\eta=1$, $\Pi_N$ is a projection matrix of rank $m$, and the projection
matrices $\{\Pi_i\}_{i\le N}$ determine a tight $N$-point simplex in
$G(m,n)$.
\end{proposition}

\begin{proof}
The automorphism $\sigma$ of $\R^n$ determines an automorphism of $G(m,n)$ by
acting simultaneously on basis vectors, and this latter automorphism is an
isometry.  The first condition states that the planes spanned by
$\{x_{i,1},\dots,x_{i,m}\}$ and $\{x_{k+i,1},\dots,x_{k+i,m}\}$ are related
by this isometry; thus, taking all $i < N$, we have $k$ orbits under the
cyclic-shift action, each of size $n$. The next two conditions ensure that
the matrices $\Pi_i$ for $i<N$ are orthogonal projections of rank $m$.  Thus
the inner products amongst them determine distances in $G(m,n)$.   Now, by
applying the cyclic-shift isometry we see that the fourth condition is
sufficient to force $\{\Pi_i\}_{i<N}$ to determine a regular simplex.  Let
$\alpha = \langle \Pi_i,\Pi_{i'} \rangle$ be its common inner product.

Consider now the matrix $\Pi_N$.  It is symmetric, being a linear combination
of symmetric matrices.  Moreover, it is cyclic-symmetric, since $\sum_{i <
N}\Pi_i$ is a sum over orbits of the cyclic shift. It follows that $\Pi_N^2 -
\eta \Pi_N$ also shares these properties. Now a matrix with cyclic-symmetry
is determined by its first row, as the other rows are just shifts thereof.  A
matrix which is also symmetric is determined by the first $\lfloor
\tfrac{n}{2} \rfloor+1$ entries in the first row. Therefore, by the last
condition, $\Pi_N^2 - \eta \Pi_N = 0$.

It follows that the eigenvalues of $\Pi_N$ are all either $0$ or $\eta$.  Let
$r$ be the rank of $\Pi_N$, so that $\Tr \Pi_N = r\eta$.  But, since $\Tr
\Pi_i = m$ for all $i < N$, we have $\Tr \Pi_N = m$.  Hence $\eta = m/r$ is
$m$ times the reciprocal of an integer.  The assumption $\eta \in
\big(\tfrac{m}{m+1},\tfrac{m}{m-1}\big)$ then forces $\eta = 1$, from which
we conclude that $\Pi_N$ is an orthogonal projection matrix of rank $m$.

Now we check that $\langle \Pi_i, \Pi_N \rangle = \alpha$ for all $i < N$.
Since $\langle \Pi_i, \Pi_i \rangle = m$ for all $i$,
\begin{equation} \label{eq:cyclicgrass}
\Pi_N = \frac{Nm}{n} I_n - \sum_{i
< N} \Pi_i,
\end{equation}
and $\langle \Pi_i,\Pi_{i'} \rangle = \alpha$ for distinct $i,i' \le N-1$, we see
that $\langle \Pi_N, \Pi_i \rangle$ is independent of $i$.  Let $\beta$ be
this common value.  Taking the inner product of \eqref{eq:cyclicgrass} with
$\Pi_N$ and $\Pi_{i'}$, we obtain
\begin{align*}
m &= \frac{Nm^2}{n} - (N-1)\beta \quad \textup{and}\\
\beta &= \frac{Nm^2}{n} - (N-2)\alpha - m,
\end{align*}
respectively.  Subtracting and canceling the (nonzero) factor of $N-2$ yields
$\alpha=\beta$. Thus, we have a regular simplex, which is tight by
\eqref{eq:cyclicgrass}.
\end{proof}

Note that the plane with projection matrix $\Pi_N$ is the unique
plane completing $\{\Pi_i\}_{i<N}$ into a tight simplex.  This plane
is a fixed point for the cyclic-shift action.

In our case of interest we found a point in which the conditions described in
Proposition \ref{prp:11inG25} are nonsingular. This yields the following.

\begin{theorem} \label{thm:11inG25}
There exists a tight $11$-point simplex in $G(2,5)$.  In fact, there is
such a tight simplex with cyclic symmetry.
\end{theorem}

We remark in passing that every approximate $11$-point tight simplex in
$G(2,5)$ we found numerically exhibited a symmetry group conjugate to the
cyclic-symmetry discussed here.  With this evidence as well as the fact that
$\sdim{11}{G(2,5)} = -2 < 0$, we conjecture that every tight $11$-point
simplex in $G(2,5)$ has a nontrivial symmetry group.

We will settle the remaining cases with algebraic constructions. The four
cases of $4$-point simplices afford constructions using only rationals and
quadratic irrationals, so we give them explicitly here. Given the provided
matrices, the proof of the following theorem consists only of a
straightforward calculation.

\begin{theorem} \label{theorem:grass4}
The four $2 \times 5$ matrices in Figure \ref{fig:4inG25} are generator
matrices whose corresponding planes form a tight simplex in $G(2,5)$; i.e.,
they have orthonormal rows and the spans of those rows constitute a tight
simplex. Similarly, the matrices in Figures \ref{fig:4inG36},
\ref{fig:4inG37}, and \ref{fig:4inG38} determine tight simplices in $G(3,6)$,
$G(3,7)$, and $G(3,8)$, respectively.
\end{theorem}

\begin{figure}
\[\begin{array}{ccc}

 \begin{pmatrix}    1 & 0 & 0 & 0 & 0  \\
  0 & 1 & 0 & 0 & 0    \end{pmatrix}

& &

\dfrac{1}{5}  \begin{pmatrix}    -\sqrt{3} &  2 &  \sqrt{6} &  2 \sqrt{3} &  0   \\
      0 &  \sqrt{3} &  -\sqrt{2} &  0 &  \sqrt{20}     \end{pmatrix}

\\ {} & {} & {} \\

\dfrac{1}{5}  \begin{pmatrix}    3 & 0 & 0 & 4 & 0  \\
     0 & 1 & -2 \sqrt{6} & 0 & 0    \end{pmatrix}

& &

\dfrac{1}{5}  \begin{pmatrix}     \sqrt{3} &  2 &  \sqrt{6} &  -2 \sqrt{3} &  0  \\
      0 &  -\sqrt{3} &  \sqrt{2} &  0 &  \sqrt{20}     \end{pmatrix}

\end{array}\]
\caption{Generator matrices for a tight $4$-point simplex in $G(2,5)$.}
\label{fig:4inG25}
\end{figure}

\begin{figure}
\[\begin{array}{ccc}

 \begin{pmatrix}    1 & 0 & 0 & 0 & 0 & 0  \\
  0 & 1 & 0 & 0 & 0 & 0  \\
  0 & 0 & 1 & 0 & 0 & 0    \end{pmatrix}

& &

 \dfrac{1}{\sqrt{2}}  \begin{pmatrix}
 0 & 0 & 0 & 0 & 0 & \sqrt{2}    \\
  0 & 1 & 0 & 0 & 1 & 0  \\
            0 & 0 & 1 & -1 & 0 & 0
           \end{pmatrix}

\\ {} & {} & {} \\

 \dfrac{1}{\sqrt{2}}  \begin{pmatrix}    1 & 0 & 0 & 0 & 0 & 1  \\
            0 & 1 & 0 & 0 & -1 & 0  \\
            0 & 0 & 0 & \sqrt{2} & 0 & 0    \end{pmatrix}

& &

 \dfrac{1}{\sqrt{2}}  \begin{pmatrix}    1 & 0 & 0 & 0 & 0 & -1  \\
            0 & 0 & 0 & 0 & \sqrt{2} & 0 \\
             0 & 0 & 1 & 1 & 0 & 0     \end{pmatrix}

\end{array}\]
\caption{Generator matrices for a tight $4$-point simplex in $G(3,6)$.}
\label{fig:4inG36}
\end{figure}

\begin{figure}
\[\begin{array}{ccc}

 \begin{pmatrix}    1 & 0 & 0 & 0 & 0 & 0 & 0  \\
  0 & 1 & 0 & 0 & 0 & 0 & 0  \\
  0 & 0 & 1 & 0 & 0 & 0 & 0    \end{pmatrix}

& &

\setlength{\arraycolsep}{2pt}\dfrac{1}{7}  \begin{pmatrix}     0 &  -2 \sqrt{2} &  3 &  2 \sqrt{3} &  2 \sqrt{5} &  0 & 0  \\
       -\sqrt{5} &  0 &  -2 \sqrt{2} &  \sqrt{6} &  0 &  \sqrt{30} &  0  \\
       0 &  -\sqrt{5} &  0 & 0 &  -\sqrt{2} &  0 &  \sqrt{42}    \end{pmatrix}

\\ {} & {} & {} \\

  \setlength{\arraycolsep}{2pt}\dfrac{1}{7}  \begin{pmatrix}
5 & 0 & 0 & 0 & 0 & 2 \sqrt{6} &  0\\
       0 & 3 &  0 & 0 & 2 \sqrt{10} & 0 & 0  \\
      0 & 0 & -1 &  4 \sqrt{3} &  0 & 0 & 0
           \end{pmatrix}

& &

\setlength{\arraycolsep}{2pt} \dfrac{1}{7}  \begin{pmatrix}     0 &  2 \sqrt{2} &  3 &  2 \sqrt{3} &  -2 \sqrt{5} &  0 & 0  \\
       \sqrt{5} &  0 &  -2 \sqrt{2} &  \sqrt{6} &  0 &  -\sqrt{30} &  0  \\
       0 &  \sqrt{5} &  0 & 0 &  \sqrt{2} &  0 &  \sqrt{42}    \end{pmatrix}

\end{array}\]
\caption{Generator matrices for a tight $4$-point simplex in $G(3,7)$.}
\label{fig:4inG37}
\end{figure}

\begin{figure}
\[\begin{array}{ccc}

  \begin{pmatrix}    1 & 0 & 0 & 0 & 0 & 0 & 0 & 0  \\
   0 & 1 & 0 & 0 & 0 & 0 & 0 & 0  \\
   0 & 0 & 1 & 0 & 0 & 0 & 0 & 0    \end{pmatrix}

& &

 \dfrac{1}{2}  \begin{pmatrix}         0 & 0 & 0 & 1 & 0 & 0 & 0 & -\sqrt{3}    \\
  0 & 1 & 0 & 0 & 0 & -\sqrt{3} & 0 & 0  \\
      0 & 0 & 1 & 0 & 0 & 0 & \sqrt{3} & 0
\end{pmatrix}

\\ {} & {} & {} \\

 \dfrac{1}{2}  \begin{pmatrix}    1 & 0 & 0 & 0 & \sqrt{3} & 0 & 0 & 0  \\
      0 & 1 & 0 & 0 & 0 & \sqrt{3} & 0 & 0  \\
      0 & 0 & 0 & 2 & 0 & 0 & 0 & 0    \end{pmatrix}

& &
 \dfrac{1}{2}  \begin{pmatrix}    1 & 0 & 0 & 0 & -\sqrt{3} & 0 & 0 & 0  \\
      0 & 0 & 0 & 1 & 0 & 0 & 0 & \sqrt{3} \\
      0 & 0 & 1 & 0 & 0 & 0 & -\sqrt{3} & 0     \end{pmatrix}

\end{array}\]
\caption{Generator matrices for a tight $4$-point simplex in $G(3,8)$.}
\label{fig:4inG38}
\end{figure}

We are now left with the cases of $7$- and $8$-point tight simplices in
$G(2,4)$.  These cases are more interesting; the simplest explicit
coordinates we have been able to find for them require algebraic numbers of
degree $4$ and $6$, respectively.  Because of this, instead of presenting the
algebraic numbers here we rely on a computer algebra system to (rigorously)
verify the calculation. The computational method is discussed in
\textsection\ref{subsec:rtrip}. Here we simply record the result.

\begin{theorem} \label{thm:7and8}
There exist $7$- and $8$-point tight simplices in $G(2,4)$.
\end{theorem}

We remark in passing that $G(2,4)$ contains tight simplices of
$N$ points for all $N \le 10$ (the theoretical maximum) except
for $N=9$.  Compared with the other spaces studied in this
paper, only the quaternionic and octonionic projective planes
have such a wealth of simplices.  Note also that there does not
seem to exist a tight simplex of size one less than the upper
bound in any of these spaces (see Conjectures~\ref{conj:no14}
and~\ref{conj:no26}).

\section{Algorithms and computational methods} \label{algs}

We used computer assistance in several different aspects of
this work.  Our main results involve two different
computational steps: finding approximate solutions and then
rigorously proving existence of a nearby solution. We also
require a method for computing with real algebraic numbers for
Theorem~\ref{thm:7and8}, and we must discuss how to compute
stabilizers of simplices and estimate the dimensions of
solution spaces. This section describes the algorithms and
programs used for each of these tasks.

\subsection{Proof certificates}
\label{subsec:proving}

Only the rigorous proof component is needed to verify our main
theorems.  Therefore, for ease of verification, we provide
\textsc{PARI/GP} code that gives a self-contained proof of
existence for each case. We chose \textsc{PARI} because it is
freely available and has support for multivariate polynomials
and arbitrary-precision rational numbers \cite{P}.  Our code
is relatively simple and straightforward to adapt to other
computer algebra systems.  It covers cases with a range of
matrix sizes, and the running times of the individual
existence proofs vary widely.  We have been able to complete
the full verification in less than a day on a 2015 personal
computer.

The existence proofs rely on Theorem~\ref{thm:implicit} via
Corollary~\ref{cor:polybounds}.  In particular, we use the
$\ell_\infty$ norm on the domain and codomain and we apply
Lemma~\ref{lma:polybounds} to bound the variation of the
Jacobian over the cube of radius $\varepsilon$.  To check the
hypotheses of Corollary~\ref{cor:polybounds}, we need to choose
$\varepsilon > 0$, the starting point $x_0$, and a matrix $T$
and then compute the operator norms of $T$ and $Df(x_0) \circ T
- \id_{\R^n}$.  We provide input files that specify our choices of
$x_0$, presented using rational numbers with denominator $2^{50}$,
as well as the constraint function $f$.  We
then compute $T$ as described later in this section. Computing operator norms
is easy, because the $\ell_\infty\to\ell_\infty$ operator norm
of a matrix is just the maximum of the $\ell_1$ norms of its
rows. (This is one of our primary reasons for choosing the
$\ell_\infty$ norm; for many choices of norms, approximating
operator norms of matrices is NP-hard \cite{HO}.) We always use
$\varepsilon = 10^{-9}$, so that the conclusion of
Corollary~\ref{cor:polybounds} is that there is an exact
solution, each of whose coordinates differs from our starting
point by less than $10^{-9}$.  In other words, the error is less than one
nanounit.

\begin{table}
\caption{Files for proof certificates.}
\label{table:files}
\begin{center}
\begin{tabular}{llll}
\toprule
\texttt{rigorous\_proof.gp} & \texttt{run\_all\_proofs.gp}\\
\midrule
\texttt{hp\_general.gp} & \texttt{hp2\_12.gp} & \texttt{hp2\_13.gp} & \texttt{hp2\_15.gp}\\
\texttt{op2\_general.gp} & \texttt{op2\_24.gp} & \texttt{op2\_25.gp} & \texttt{op2\_27.gp}\\
\texttt{grass\_general.gp} & \texttt{grass2\_5\_11.gp}\\
\midrule
\texttt{projective\_data.txt} & \texttt{grass\_data.txt}\\
\bottomrule
\end{tabular}
\end{center}
\end{table}

These calculations are organized into fourteen files, enumerated
in Table~\ref{table:files}.  All of these files are available
by downloading the source files for this paper from the
\texttt{arXiv.org} e-print archive.  The file
\texttt{rigorous\_proof.gp} implements
Corollary~\ref{cor:polybounds},
and \texttt{run\_all\_proofs.gp} then proves our
existence results using the remaining files for input.  The
next ten files in Table~\ref{table:files} describe the
constraints in each of our applications
(Propositions~\ref{prp:general}, \ref{prp:12}, \ref{prp:13},
\ref{prp:15}, \ref{prp:generalpanda}, \ref{prp:24},
\ref{prp:25}, \ref{prp:27}, \ref{prp:grass}, and
\ref{prp:11inG25}, respectively).
Finally, the last two files specify the
starting points, i.e., explicit numerical
approximations for the simplices.

The translation from mathematics to computer algebra code is
straightforward, with just a few issues to address.  One is that
in the cases with cyclic symmetry, some variables are constrained to be equal
to others (for example, the coordinates of $x_{m+i}$ are a
cyclic shift of those of $x_i$ in Proposition~\ref{prp:12}).\footnote{Here and in the next
paragraph, $x_i$ is not to be confused with the starting point $x_0$.}
Our data files contain all the points, but in the proofs we
eliminate these redundant variables for the sake of efficiency.  For
example, \texttt{projective\_data.txt} specifies $12$ points in
$\HH\Proj^2$, and \texttt{hp2\_12.gp} ignores all but the first
four of them.

Another issue is that in three cases
(Propositions~\ref{prp:15}, \ref{prp:27}, and
\ref{prp:11inG25}) we require certain quantities to be close
to $1$.  For example, in Proposition~\ref{prp:15} we need
$\big||x_i|^4 - 1\big|$ to be at most $10^{-6}$ for each $i$.  This could
easily be checked by direct computation using the $10^{-9}$
bound for distance from the starting point, but it is simpler
to use the following trick. For each $i$, we add a new variable
$v_i$, add a new constraint $v_i = |x_i|^4$, and initialize
$v_i$ to be $1$ at the starting point.  Then we can conclude
that $\big||x_i|^4 - 1\big| < 10^{-9}$ in the exact solution
with no additional computation.

All that remains is to describe how we compute the approximate
right inverse $T$ for use in Theorem~\ref{thm:implicit}.  Let
$J$ be the Jacobian $Df(x_0)$, which by assumption has full row
rank.  A natural choice for $T$ would be the least-squares right
inverse $J^t (J J^t)^{-1}$ of $J$, but inverting matrices using exact
rational arithmetic is slow and the denominators become large.
To save time, we approximate $J^t (J J^t)^{-1}$ using
floating-point arithmetic and obtain $T$ by rounding it to
a rational matrix with denominator $2^{50}$.  Once we have $T$,
the proof is then carried out using only exact rational
arithmetic and is therefore completely rigorous.

The use of floating-point arithmetic to obtain $T$ raises
one concern about reproducibility.  Floating-point error
depends delicately on how a computation is carried out, so
using a different computer algebra system (or even a different
version of \textsc{PARI/GP}) might give a slightly different
matrix $T$, which could in principle prevent the proof from
being verified. To guarantee reproducibility, we have analyzed
how close an approximation to $J^t(J J^t)^{-1}$ is needed to
make the proof work:
in each of our existence proofs, every $T$ satisfying
$\big|\big|T - J^t(J J^t)^{-1}\big|\big| < 10^{-2}$ works.
Any floating-point computation to produce $T$ will meet this
undemanding bound if the working precision is high enough,
and we have found the default \textsc{PARI} precision to be more
than sufficient.

To check this bound of $10^{-2}$, first suppose we have some
matrix $T$ that works in Corollary~\ref{cor:polybounds}.
Examining the slack in the corollary's hypotheses gives an
explicit bound
\[
\delta_0 = \frac{1 - ||T|| \cdot
|f(x_0)|/\varepsilon - ||JT - I_n || - \varepsilon \, |f| d (d-1)
\eta^{d-2} ||T||}{||J|| + \varepsilon \, |f| d (d-1)
\eta^{d-2} + |f(x_0)|/\varepsilon}
\]
such that we can replace $T$ with an arbitrary $T'$ satisfying
$||T'-T|| < \delta_0$.  Now every $T'$ satisfying
\[
\big|\big|T' - J^t(J J^t)^{-1}\big|\big| < \delta
\]
works as long as $\delta \le \delta_0 - \big|\big|T - J^t(J
J^t)^{-1}\big|\big|$.  We concluded that $\delta = 10^{-2}$ works by
examining all of our cases and applying
the following lemma to bound the quantity $\big|\big|T - J^t(J
J^t)^{-1}\big|\big|$ from above.

\begin{lemma}
Suppose $J \in \R^{n \times m}$ and $T \in \R^{m \times n}$,
and let $||\cdot||$ denote the operator norm with respect to
some choice of norms on $\R^n$ and $\R^m$. If $||I_n - T^t T J J^t|| <
1$, then $J$ has full row rank and
\[
\big|\big|T - J^t(J J^t)^{-1}\big|\big| \le \frac{||T J J^t - J^t|| \cdot ||T^t T||}{1 - ||I_n - T^t T J J^t||}.
\]
\end{lemma}

Note that this bound is reasonably natural: if $T = J^t(J J^t)^{-1}$, then $T J J^t = J^t$ and
$T^t T J J^t = I_n$, so the bound vanishes.

\begin{proof}
For all $A,B \in \R^{n \times n}$ with $||I_n - AB|| < 1$, $B$
is invertible and
\[
\big|\big| B^{-1} \big|\big| \le \frac{||A||}{1-||I_n - AB||},
\]
because we can take
\[
B^{-1} = \sum_{i \ge 0} (I_n - AB)^i A.
\]
Setting $A=T^t T$ and $B = J J^t$ shows that $J J^t$ is invertible
(so $J$ has full row rank) and
\[
\big|\big| (J J^t)^{-1} \big|\big| \le \frac{||T^t T||}{1 - ||I_n - T^t T J J^t||}
\]
Now combining this estimate with
\[
\big|\big|T - J^t(J J^t)^{-1}\big|\big| \le ||T J J^t - J^t|| \cdot \big|\big|(J J^t)^{-1}\big|\big|
\]
completes the proof.
\end{proof}

Finally, we note in passing that the implementation of our proof techniques in
\texttt{rigorous\_proof.gp} is general enough to apply to a range of problems.
For example, we have used it to reproduce the results of \cite{ChenWomersley}
and to prove some of the conjectures in \cite{HS96}, such as the existence
of a $26$-point $6$-design in $S^2$.  (Handling all of the conjectures in \cite{HS96}
would require additional ideas, perhaps along the lines of the
special-case arguments in
\textsection\ref{subsec:1213} and \textsection\ref{subsec:15}.)

\subsection{Finding approximate solutions}
\label{subsec:finding}

To find approximate solutions we used a new computer package
called \textsc{QNewton}, which was written by the last named author
and can be obtained from him upon request.
\textsc{QNewton} consists of a \Cplusplus{} library with a
\textsc{Python} front end and is designed to find solutions
to polynomial equations over real algebras.  Furthermore,
\textsc{QNewton} can rigorously prove existence of solutions
using Theorem~\ref{thm:implicit}.

We have chosen to use both \textsc{QNewton} and \textsc{PARI/GP}
because they have different advantages:
the \textsc{PARI/GP} code is shorter and easier to check or
adapt to other computer algebra systems, while
\textsc{QNewton} provides a flexible, integrated environment
for both computing approximate solutions and proving existence.

After we specify the polynomials and constraints for the
problem and an initial point, \textsc{QNewton} attempts to find
a solution using a damped Newton's method algorithm. Newton's
method converges rapidly in a neighborhood of a solution, but
it is ill-behaved away from solutions; thus we damp the steps
so that no coordinate changes in a single step by more than a
specified upper bound.

Because the codes we seek are energy minimizers, another
approach to find them would have been gradient descent.  In
practice, we have found that gradient descent is much slower
than Newton's method.

In our computations, we used random Gaussian variables for the
initial points and a maximum step size of $0.1$.  Because our
variables represent unit vectors, the step size is
approximately one order of magnitude less than the natural
scale.  By using this approach we were able to find a solution
in all cases in which we think there should exist one, using
just a few different random starting positions.  In most cases
we found a solution on the first try. These approximate
calculations use double-precision floating-point arithmetic, so
we can only expect convergence up to an error of approximately
$10^{-15}$. In all cases this was more than sufficient for our
goals of rigorous proof.

Suppose that, as in Theorem \ref{thm:implicit}, we are solving
for a zero of a function $f \colon \R^m \rightarrow \R^n$.
Newton's method calls for taking repeatedly taking steps
$\Delta x$ satisfying $Df(x) \cdot \Delta x = -f(x)$.  In
particular, we must repeatedly solve linear systems. When $m
> n$ the system is underdetermined. Also, $Df(x)$ may fail to be
surjective. Hence we need a linear solver tolerant of such
problems. \textsc{QNewton} uses a \emph{least-squares} solver
that treats small singular values of $Df(x)$ as zero;
specifically, it uses the \textsc{dgelsd} function in
\textsc{LAPACK} \cite{LAPACK}. By using such a solver we can
handle cases with redundant constraints.  This was particularly
useful when we were first determining a minimal set of
constraints for our problems.

\textsc{QNewton} has native support for multiplication in $\R$,
$\C$, $\HH$, and $\OO$. Also, it uses automatic (reverse)
differentiation to compute the Jacobian of the constraint
function. These two features substantially increase its
performance.

The \textsc{QNewton} package also has a mechanism for
computer-assisted proof using Theorem \ref{thm:implicit}. Like
the proofs discussed in the previous section, it uses the
$\ell_\infty$ norm on both domain and codomain. However, unlike
those proofs, \textsc{QNewton} does not use rational
arithmetic, nor does it use Lemma \ref{lma:polybounds} to
control the variation of the Jacobian.  Instead it uses
\emph{interval arithmetic}.

Interval arithmetic is a standard tool in numerical analysis to control the
errors inherent in floating-point computations. The principle is simple:
instead of rounding numbers so that they are exactly representable in the
computer, we work with intervals that are guaranteed to contain the correct
value.  For instance, consider a hypothetical computer capable of storing 4
decimal digits of precision.  Using floating-point arithmetic, $\pi$ would
best be represented as $3.142$.  Using this, if we computed $2 \cdot  \pi$
then we would get $6.284$, which is obviously not correct.  By contrast,
interval arithmetic on the same computer would represent $\pi$ as the
interval $[3.141,3.142]$.  Then $2 \cdot \pi$ would be represented by the
interval $[6.282,6.284]$, which does contain the exact value.

It is clear that balls with respect to the $\ell_\infty$ norm can be
naturally represented using interval arithmetic. Thus, in the notation of
Theorem \ref{thm:implicit}, for each entry of the Jacobian matrix we can
easily compute an interval that contains this entry of $Df(x)$ for every $x
\in B(x_0,\varepsilon)$.  We then compute an interval guaranteed to contain
$||Df(x) \circ T - \id_{\R^n}||$ for all such $x$, and an interval guaranteed
to contain $1 - ||T|| \cdot |f(x_0)|/\varepsilon$. If the upper bound of the
first interval is less than the lower bound of the second interval, then we
are assured that Theorem \ref{thm:implicit} applies.

\textsc{QNewton} uses the \textsc{Mpfi} library to provide support for
interval arithmetic \cite{RR}.  That in turn relies on \textsc{Mpfr}, a
library for multiple-precision floating-point arithmetic \cite{MPFR}. One of
the main problems with interval arithmetic is that the size of the intervals
can grow exponentially with the number of arithmetic operations; this problem
can be ameliorated by increasing the precision of the underlying
floating-point numbers.  It was not an issue in our applications, though.

Finally, we remark upon the computation of the matrix $T$. It
is supposed to be approximately a right inverse of $Df(x_0)$,
but otherwise we are free in choosing it.  In \textsc{QNewton},
we compute $T$ much as in the \textsc{PARI} code.  First we compute the
matrix $Df(x_0)$ approximately, using floating-point
arithmetic.  Then we find its pseudoinverse (i.e., the
least-squares right inverse), again using inexact
floating-point arithmetic. Finally, we take the result and
replace it with intervals of width $0$.  This approach is fast
and, since $T$ need not be the exact pseudoinverse, still gives
rigorous results. It is possible to compute $Df(x_0)$ in
interval arithmetic and then compute the pseudoinverse in the
same way; this is a bad idea, though, because inverting a
matrix in interval arithmetic can result in very large
intervals.

\subsection{Finding stabilizers}
\label{subsec:stabilizers}

In all but one case, namely $5$-point simplices in $\OO\Proj^2$, our reported
local dimension for the moduli space of tight simplices has the dimension of the full
symmetry group deducted. That is valid when each simplex in a neighborhood of the point
under consideration has finite (i.e., zero-dimensional) stabilizer.  This is an open
condition and thus only needs to be checked at that single point. We checked this condition
by (i) finding a basis for the Lie algebra of the symmetry group, (ii) applying each
element of that basis to the points of the simplex to produce tangent vectors, and (iii) verifying that the resulting
vectors are linearly independent.  In the remainder of this subsection, we explain the
calculations in more detail.

The relevant symmetry groups are $\operatorname{Sp}(d)/\{\pm1\}$ for
$\HH\Proj^{d-1}$ and $F_4$ for $\OO\Proj^2$, which have dimensions
$2d^2 + d$ and $52$, respectively.  Let $K$ be $\HH$ or $\OO$, as appropriate,
and let $\mathfrak{g}$ be the
Lie algebra of the isometry group of $K\Proj^{d-1}$ (i.e.,
$\mathfrak{g} = \mathfrak{sp}_d$ if $K=\HH$ and
$\mathfrak{g} = \mathfrak{f}_4$ if $K=\OO$).

The Lie algebra $\mathfrak{g}$ acts on the space $\mathcal{H}(K^d)$ of Hermitian matrices.
In fact, in this representation it is generated
by commutation with
traceless skew-Hermitian matrices and application of
derivations of the underlying algebra $\HH$ or $\OO$ (see \cite{Ba}).
The Lie algebra of the stabilizer of the simplex annihilates the projection
matrices for the simplex.
Thus, if the dimension of the $\mathfrak{g}$-orbit in $\mathcal{H}(K^d)^N$ of an
$N$-point simplex is at least $D$, then the dimension of the stabilizer is at most $\dim_\R \mathfrak{g} - D$.

It remains to compute a lower bound for the dimension of the $\mathfrak{g}$-orbit
of the simplex determined by unit vectors $x_1,\dots,x_N \in K^d$.
However, we do not have explicit vectors for the points in the simplex.
Instead, we have approximations $\tilde x_1,\dots,\tilde x_N \in K^d$.
These vectors
are $\varepsilon$-approximations under the $\ell_\infty$ norm with respect
to the standard real basis of $K^d$, where $\varepsilon = 10^{-9}$
(see \textsection\ref{subsec:proving}), and we will give a lower bound that
holds over the entire $\varepsilon$-neighborhood of $(\tilde x_1,\dots,\tilde x_N)$.
When we refer below to real entries of vectors and
matrices, we will use the standard real basis of $K$; thus, each entry over $K$
comprises $\dim_\R K$ real entries.

Before applying $\mathfrak{g}$, we must convert the vectors $x_i$ to projection matrices.
To bound the approximation error, note
that each real entry of $x_i$ is bounded by $1$ in absolute value (since $x_i$ is a unit vector), and thus each real entry of
$\tilde x_i$ is bounded by $1+\varepsilon$.  It follows that
the real entries of $\widetilde\Pi_i := \tilde x_i \tilde x_i^\dagger$ approximate those of the true
projection matrices $\Pi_i := x_i x_i^\dagger$ to within $(2\varepsilon+\varepsilon^2)\dim_\R K$, because
each entry over $K$ is just a product in $K$ (i.e., each real entry is the sum of $\dim_\R K$ real products) and
\begin{equation} \label{eq:uv}
|uv - \tilde u \tilde v| \le |u-\tilde u|\cdot|v| + |\tilde u| \cdot |v-\tilde v|
\end{equation}
for $u,v,\tilde u, \tilde v \in \R$.

To understand the action of $\mathfrak{g}$ on $\Pi_1,\dots,\Pi_N$, we begin by
choosing a basis of $\mathfrak{g}$.  For each basis element, applying it
to each of $\Pi_1,\dots,\Pi_N$ and then concatenating the real entries of these $N$ Hermitian matrices
yields a single vector of dimension $k := N d^2 \dim_\R K$.  The resulting vectors form
a $(\dim_\R \mathfrak{g}) \times k$ real matrix $M$,
and the rank of $M$ is the dimension of the $\mathfrak{g}$-orbit.  Of course, the
difficulty is that all we can compute is the approximation $\Mtilde$ to $M$
obtained from $\widetilde \Pi_1, \dots, \widetilde \Pi_N$.  Each entry of $\Mtilde$ is
within $\delta$ of the corresponding entry of $M$, where $\delta$ is
$(2\varepsilon+\varepsilon^2)\dim_\R K$ times the greatest $\ell_\infty \to \ell_\infty$
operator norm (with respect to the standard real basis of $\mathcal{H}(K^d)$) of any basis element of $\mathfrak{g}$.

\begin{lemma}
Let $M$ and $\Mtilde$ be $m \times k$ real matrices whose entries differ by at most $\delta$.
Then the rank of $M$ is at least the number of eigenvalues of $\Mtilde  \Mtilde^t$
that are greater than $m k \delta \big(2 \max_{i,j} \big|{\Mtilde_{i,j}}\big| + \delta\big)$.
\end{lemma}

\begin{proof}
One can check using \eqref{eq:uv} that
the entries of $MM^t$ and $\Mtilde \Mtilde^t$ differ by at most
$\gamma := k \delta \big(2 \max_{i,j} \big|{\Mtilde_{i,j}}\big| + \delta\big)$.
Let $V$ be the span of the eigenvectors of $\Mtilde \Mtilde^t$
with eigenvalues greater than $\gamma m$.  For all $v \in V$ with $\ell_2$ norm $|v|_2 = 1$, we have
\[
v^t \Mtilde \Mtilde^t v > \gamma m.
\]
On the other hand, $|v|_1 \le \sqrt{m}$ by the Cauchy-Schwarz inequality.
Using the observation that
\begin{equation} \label{eq:l1linfty}
|\langle a,b \rangle| \le |a|_1 |b|_\infty
\end{equation}
for vectors $a$ and $b$, it readily follows that
\[
\big|\big(MM^t-\Mtilde \Mtilde^t\big) v\big|_\infty \le \gamma \sqrt{m}.
\]
Applying \eqref{eq:l1linfty} once more, we obtain
\[
\big|v^t\big(MM^t-\Mtilde \Mtilde^t\big) v\big| \le \gamma m
\]
and hence
\[
v^t M M^t v > 0.
\]
We have shown that the restriction of $M M^t$ to $V$ is positive definite.
Therefore $\rank M = \rank M M^t \ge \dim V$, as desired.
\end{proof}

To apply this lemma, we simply compute the characteristic polynomial of $\Mtilde \Mtilde^t$.
Its roots are the eigenvalues of $\Mtilde \Mtilde^t$ with multiplicity, and we apply
Sturm's theorem to count those that
are greater than $m k \delta \big(2 \max_{i,j} \big|{\Mtilde_{i,j}}\big| + \delta\big)$.  All of these computations use
exact rational arithmetic and thus yield a rigorous lower bound for the rank of
$M$, which is the dimension of the $\mathfrak{g}$-orbit of the simplex
$\Pi_1,\dots,\Pi_N$.  In other words, they give a rigorous upper bound for the dimension
of the stabilizer.

We have implemented these calculations in \textsc{PARI/GP}, and the code can be obtained
as described in \textsection\ref{subsec:proving}.  The file
\texttt{apply\_lie\_basis.gp} sets up the machinery, and \texttt{stabilizers.gp}
applies it to show that all of the projective simplices we have found have zero-dimensional
stabilizers, except for $5$ points in $\OO\Proj^2$.  In that exceptional case,
the stabilizer has dimension at most $3$. This is good enough because,
translated into a dimension for the moduli space of simplices, that bound says that
the dimension is at most $0$; hence the dimension must equal $0$.

\subsection{Real algebraic numbers}
\label{subsec:rtrip}

To verify equations involving algebraic numbers of moderately high degree, we
require a computational method for rigorously doing basic arithmetic with
such numbers. One possibility is to work in a single number field, but even
when each number we manipulate is of manageable degree, the smallest field
containing them all may have exponentially high degree. We will instead use
the standard approach of ``isolating intervals,'' which is implemented in
many modern computer algebra systems. There is no explicit support for the
isolating interval method in \textsc{PARI/GP}, so in order to present all of
our computer files in one system we provide a short implementation in
addition to the pertinent data files for our applications.

The technique is as follows.  A real algebraic number $\alpha$ is represented
by a triple $(p(x),\ell,u)$, where $p(x)$ is a polynomial with integer
coefficients such that $p(\alpha)=0$, $\ell$ and $u$ are rational numbers
such that $\alpha \in [\ell,u]$, and $p(x)$ has a unique root in the interval
$[\ell,u]$ (namely, $\alpha$). We always take $p(x)$ to be (a scalar multiple
of) the minimal polynomial of $\alpha$, and we use Sturm sequences to
rigorously count the number of real roots in a given interval. Given
representations $(p_{\alpha},\ell_\alpha,u_\alpha)$ and
$(p_{\beta},\ell_\beta,u_\beta)$ for two real algebraic numbers
$\alpha,\beta$, we compute a representation for $\alpha+\beta$ by first
taking the resultant, in the variable $t$, of the polynomials $p_\alpha(t)$
and $p_\beta(x-t)$. This gives a polynomial in $x$ for which $\alpha+\beta$
is a root.  We then factor the resulting polynomial and count the number of
roots for each irreducible factor in the interval
$[\ell_\alpha+\ell_\beta,u_\alpha+u_\beta]$.  If there is more than one
factor that has a root in that interval or some factor has multiple roots,
then we bisect the starting intervals $[\ell_\alpha,u_\alpha]$ and
$[\ell_\beta,u_\beta]$, using Sturm sequences for $p_\alpha$ and $p_\beta$ to
choose the halves containing $\alpha$ and $\beta$, respectively.  After a
finite number of steps we are left with a valid representation for
$\alpha+\beta$.  Computing a representation for $\alpha \cdot \beta$ proceeds
similarly, beginning with the resultant of $p_\alpha(t)$ and $t^{\deg
p_\beta} p_\beta(x/t)$.

Using this system, we can now elucidate the proof of existence for
$7$- and $8$-point tight simplices in $G(2,4)$.

\begin{proof}[Proof of Theorem \ref{thm:7and8}]
We provide isolating interval representations for the entries
of the $4 \times 4$ projection matrices $\Pi_1,\dots,\Pi_N$ for
the $N = 7$ or $8$ points in each simplex. To verify the
construction we need only perform a few calculations.  First we
need to check that each provided matrix $\Pi$ satisfies $\Pi =
\Pi^t$, $\Pi^2 = \Pi$, and $\Tr \Pi = 2$, as together these
conditions imply that $\Pi$ is an orthogonal projection onto a
plane.  Then we just need to verify that $\Tr \Pi_i\Pi_j =
(N-2)/(N-1)$ for $i < j \le N$. These calculations are
straightforward given our implementation of the isolating
interval method.
\end{proof}

The computer files can be obtained as described in
\textsection\ref{subsec:proving}.  The file \texttt{rtrip.gp}
implements isolating intervals (``rtrip'' refers to the representation of real
algebraic numbers using triples). Using this implementation,
\texttt{G2\_4\_verify.gp} carries out the computations with
projection matrices taken from \texttt{G2\_4\_data.txt}.

\subsection{Estimating dimensions}
\label{subsec:dimensions}

In Conjectures \ref{conj:dim12and13},
\ref{conj:dim24and25}, and
\ref{conj:dim39}, we conjecture the dimension
of certain solution spaces; here we describe the basis
for those conjectures.

Suppose, as is the case in our examples, that we are studying the zero set
$Z$ of some function $f$. Suppose moreover that we have a procedure for
converging to zeros of $f$, using, for example, Newton's method with
least-squares solving to handle degeneracy. Thus we have the ability to
generate points on $Z$, and we wish to use that ability to calculate its
dimension.  This is a simple instance of \emph{manifold learning}, the
problem of describing a manifold given sample points embedded in some
higher-dimensional space.

For our purposes we use following heuristic.  Fix $\varepsilon > 0$. Starting
with a solution $x_0$, we compute  $N$ nearby solutions $x_1,\dots,x_N$ as follows.
We first set $x'_i = x_0 + \varepsilon g_i$, where $g_i$ is a vector
of standard normal random variables, and then use our iterative solver to find a
zero $x_i$ of $f$ near $x'_i$.  To first order in $\varepsilon$, the
vectors $(x_i - x_0)/|x_i - x_0|$ are random (normalized) samples from the
tangent space of $Z$ at $x_0$. We then form the matrix whose rows are those
$N$ vectors and compute its singular values. There should be $d$ singular
values of order approximately $1$, where $d$ is the dimension of $Z$.  The
remaining singular values should be smaller by a factor of $\varepsilon$.

This procedure is certainly not rigorous, but in suitably nice cases, and
with proper choice of parameters, one can have a fair amount of confidence in
the result.  In particular, $N$ should be at least as large as the dimension
$d$ and $\varepsilon$ should be chosen small enough that, in a ball of radius
$\varepsilon$, $Z$ is well-approximated by its tangent space.  One pitfall to
avoid is that, while $\varepsilon$ needs to be small for the tangent space
approximation, it should also be large enough that the precision of the
solver is better than (approximately) $\varepsilon^2$.  If this is violated
then we may erroneously identify extra null vectors of $Df(x_0)$ as elements
of the tangent space.

In our applications we used $N=1000$ and $\varepsilon = 10^{-3}$ and we
required that Newton's method converge to within $10^{-12}$.  It was usually
easy to identify the jump in singular values after the $d$ corresponding to
the tangent space.  For instance, Conjecture \ref{conj:dim12and13} says that,
before accounting for overcounting and symmetries, we conjecture a
$66$-dimensional space of $12$-point tight simplices in $\HH\Proj^2$. This is
based on the following observation: when we ran the procedure just discussed,
the first $66$ singular values were all in the interval $[2,6]$, but the
$67$th was $0.04139564$.

\begin{remark} \label{rmk:sic-povm-dim}
Based on similar computations, we conjecture that
the moduli space of SIC-POVMs, simplices of $d^2$ points in $\C\Proj^{d-1}$, has
dimension $1$ when $d=3$ and $0$ when $d \ge 4$.  In particular, we
conjecture that, except in $\C\Proj^2$, SIC-POVMs are isolated.  This is in
accordance with the numerical results in \cite{SG}, although they searched
only for SIC-POVMs that are invariant under the Weyl-Heisenberg group.
\end{remark}

\section{Explicit constructions} \label{mub}

With the exception of Theorems~\ref{theorem:grass4} and~\ref{thm:7and8}, all
of the new results we have presented so far involve computer-assisted proofs
using Theorem~\ref{thm:implicit}. This allowed us to sidestep explicit
constructions, and it also gave local dimensions as a collateral benefit.
When an explicit construction is available, though, it can sometimes give
insight not proffered by a general existence theorem. We conclude the paper
with a few examples of this.

\subsection{Two universal optima in $\SO(4)$}

Most results in the literature concerning universal optima in continuous
spaces are set in two-point homogeneous spaces, i.e., spheres and projective
spaces. We have already seen another family of spaces (namely, real
Grassmannians) but there are many others.

Consider the special orthogonal group $\SO(n)$, endowed with the chordal
distance $d_c(U_1,U_2) = ||U_1 - U_2||_F$ coming from the embedding $\SO(n)
\hookrightarrow \R^{n^2}$ as $n \times n$ matrices equipped with the Frobenius
norm. This is not the Killing metric, but it has the advantage that its square
is a smooth function on $\SO(n) \times \SO(n)$.  Note that every element of $\SO(n)$ has
norm $n$, so up to this scaling factor we have an embedding into $S^{n^2-1}$.

By a \emph{universally optimal} code in $\SO(n)$, we mean a code that
minimizes energy for every completely monotonic function of squared chordal
distance (see \cite{CK}).  In this section we present two particularly
attractive universal optima in $\SO(4)$.

\begin{theorem} \label{thm:17}
There is a $17$-point code in $\SO(4)$ with the following properties:
it is a regular simplex, it is universally optimal, and it has a transitive symmetry group.
Moreover, there is no larger regular simplex in $\SO(4)$.
\end{theorem}
\begin{proof}
Given $a,b \in \Z/17\Z$, define the rotation matrix
\[R_{a,b} = \begin{pmatrix}
 \cos(a \theta) & -\sin(a \theta) & 0 & 0 \\
\sin(a \theta) & \cos(a \theta) & 0 & 0 \\
0 & 0 & \cos(b \theta) & -\sin(b \theta) \\
0 & 0 & \sin(b \theta) & \cos(b \theta) \end{pmatrix},\]
where $\theta = 2\pi/17$.  For any $a,b,c,d$, not all zero,
the map $\sigma_{a,b,c,d} \colon \SO(4) \rightarrow \SO(4)$
defined by $X \mapsto R_{a,b} X R_{c,d}$ is an isometry of $\SO(4)$
of order $17$.  Set
\[X_0 = \begin{pmatrix}
0 & 0 & 0 & 1 \\
0 & 1 & 0 & 0 \\
0 & 0 & 1 & 0 \\
1 & 0 & 0 & 0 \end{pmatrix} \in \SO(4)\] and let $\{X_i = R_{1,3}^i X_0
R_{4,5}^i\} \in \SO(4)$ be the orbit of $X_0$ under $\sigma_{1,3,4,5}$.  This
is a $17$-point code which, by construction, has a transitive symmetry group.
Moreover, direct calculation shows that it forms a regular simplex.

By virtue of the Euclidean embedding $\SO(4) \hookrightarrow S^{15}$, there
can be no regular simplices of more than $17$ points, and a $17$-point
regular simplex must be universally optimal (indeed, it is even universally
optimal as a code on the sphere). That proves the remaining claims of the
theorem.
\end{proof}

\begin{theorem} \label{thm:32}
There is a $32$-point code in $\SO(4)$ with the following properties: it is a
subgroup, it is universally optimal, and it forms the vertices of a
cross polytope in $S^{15}$.
\end{theorem}

\begin{proof}
The code consists of all matrices of the form
\[
\begin{pmatrix} a & 0 & 0 & 0\\
0 & b & 0 & 0\\
0 & 0 & c & 0\\
0 & 0 & 0 & d\end{pmatrix},\
\begin{pmatrix} 0 & a & 0 & 0\\
b & 0 & 0 & 0\\
0 & 0 & 0 & c\\
0 & 0 & d & 0\end{pmatrix},\
\begin{pmatrix} 0 & 0 & a & 0\\
0 & 0 & 0 & b\\
c & 0 & 0 & 0\\
0 & d & 0 & 0\end{pmatrix},
\textup{ or }
\begin{pmatrix} 0 & 0 & 0 & a\\
0 & 0 & b & 0\\
0 & c & 0 & 0\\
d & 0 & 0 & 0\end{pmatrix},
\]
where $a,b,c,d = \pm 1$ with an even number of $-1$'s.  In other words, we
use signed permutation matrices in which the underlying permutation is either
trivial or a product of disjoint $2$-cycles and the number of minus signs is
even.  It is not difficult to check that this defines a subgroup of $\SO(4)$.

The supports of these four types of matrices are disjoint, so the
corresponding points in $\R^{16}$ are orthogonal.  The inner product between
two matrices of the same type is simply the inner product of the vectors
$(a,b,c,d)$, which is $0$ or $\pm4$ because of the even number of $-1$'s.
Thus, the code forms a cross polytope in $S^{15}$.

As in Theorem \ref{thm:17}, universal optimality of $\mathcal{C}$ in $\SO(4)$
follows from universal optimality as a subset of $S^{15}$ (see \cite{CK}).
\end{proof}

\subsection{$39$ points in $\OO \Proj^2$}

\begin{theorem} \label{thm:39}
There exists a tight code $\sC$ of $39$ points in $\OO\Proj^2$.
It consists of $13$ orthogonal triples such that, for any two
points $x_i,x_j$ in distinct triples, $\rho(x_i,x_j) =
\sqrt{2/3}$. In other words, if $\Pi,\Pi'$ are the projection
matrices corresponding to two distinct points in $\sC$, then
$\langle \Pi,\Pi' \rangle$ equals $0$ if the two points are in
the same triple and otherwise equals $1/3$.
\end{theorem}

\begin{proof}
First we recall from \cite[p.~127]{Co} the standard construction of a
$12$-point universal optimum in $\C\Proj^2$: in terms of unit-length
representatives, it consists of the standard basis
\[(1,0,0),\ (0,1,0),\ (0,0,1)\]
together with the $9$ points
\begin{equation} \label{eq:omegapoints}
\frac{1}{\sqrt{3}} (1,\omega^a,\omega^b),
\end{equation}
where $\omega = e^{2\pi i/3}$ and $a,b = 0,1,2$.

To construct the desired code, we will use the standard basis together with
four rotated copies of \eqref{eq:omegapoints}. More precisely, let
$\{1,i,j,k\}$ be the standard basis of $\HH$ and let $\ell$ be any one of the
remaining four standard basis elements of $\OO$. We identify $\omega \in \C$
as an element of $\Span \{1,i\} \subset \OO$. Set $n = j\ell$.  Then we
define $\sC \subset \OO\Proj^2$ to be the code obtained from the standard
basis and the points
\begin{equation} \label{eq:39points}
\begin{matrix}
(1,\omega^a,\omega^b)/\sqrt{3}, &
(1,\omega^a j,\omega^b \ell)/\sqrt{3},\\
(1,\omega^a \ell,\omega^b n)/\sqrt{3}, &
(1,\omega^a n,\omega^b j)/\sqrt{3}
\end{matrix}
\end{equation}
for $a,b = 0,1,2$. Direct computation shows that this code has the desired
distances.  In particular, the code splits into $13$ distinguished
triples of points: the standard basis yields one such triple, and each of
the four types of points in \eqref{eq:39points} yields three triples
according to the value of $a+b$ modulo $3$.

The sums over $\sC$ of the first and second harmonics
\begin{align*}
P_1^{(7,3)}(2t-1) &= 12t - 4,\\
P_2^{(7,3)}(2t-1) &= 91t^2 - 65t + 10
\end{align*}
of $\OO\Proj^2$ both vanish; thus $\sC$ is a $2$-design.
As it has only two inner products between distinct
points, and one of those is $0$, it is tight \cite{Le1}
and in fact universally optimal \cite{CK}.
\end{proof}

The code $\sC$ in Theorem~\ref{thm:39} is a system of $13$
mutually unbiased bases.  It follows easily from linear
programming bounds that it is the largest such system possible
in $\OO\Proj^2$.

This code is not unique: we can deform it to a four-dimensional family of
tight codes by replacing $\ell$, $n$, $n$, and $j$ in the second line of
\eqref{eq:39points} with $\xi_1 \ell$, $\xi_2 n$, $\xi_3 n$, and $\xi_4 j$,
where $\xi_1,\dots,\xi_4$ are complex numbers of absolute value $1$.  The
group of isometries of $\OO\Proj^2$ fixing the remaining 21 unchanged points
is zero-dimensional (see \textsection\ref{subsec:stabilizers}, for instance),
so we have a four-dimensional family even modulo the
action of the isometry group $F_4$ of $\OO\Proj^2$.  We think the actual
space of tight codes is much larger, though. On the basis of numerical
evidence (see \textsection\ref{subsec:dimensions}), we conjecture the
following.

\begin{conjecture} \label{conj:dim39}
In a neighborhood of the code constructed in
\eqref{eq:39points}, the space of tight $39$-point codes,
modulo the action of $F_4$, is a manifold of dimension $24$.
\end{conjecture}

At present this remains just a conjecture, though, as we have been unable to
identify a nonsingular system of equations to which we can apply
Theorem~\ref{thm:implicit}.

The existence of a code of this form was conjectured by Hoggar \cite[Table
2]{Ho2} after classifying the permissible parameters for strongly regular
graphs.  Excepting a hypothetical $26$-point tight simplex, which we
conjecture does not exist, there are no remaining cases in which the
existence of a tight code in $\OO\Proj^2$ is conjectured but not resolved. In
fact, based on computations of optimal quasicodes (two-point correlation
functions subject to linear programming bounds \cite{CZ}), we are confident
there are no other tight codes in $\OO\Proj^2$ with at most $10^4$ points. We
believe there are no more of any size.

\end{document}